\documentclass[a4paper,10pt,reqno]{amsart} %
\usepackage{latexsym,amssymb}
\usepackage{hyperref}
\usepackage{amsmath,amsthm,amsxtra}
\usepackage{amsfonts}
\usepackage[american]{babel}
\usepackage{mathrsfs}
\usepackage{psfrag}
\usepackage{epsfig,inputenc}
\usepackage{graphpap,latexsym,epsf}
\usepackage{color}
\usepackage{amssymb,eucal,paralist,color,enumerate}
\usepackage{graphicx}
\definecolor{ddcyan}{rgb}{0,0.1,0.9}
\definecolor{ddmagenta}{rgb}{0.8,0,0.8}
\definecolor{orange}{rgb}{0.6,0.2,0}

\newcommand{\bele}{\begin{lemm}}
\newcommand{\enle}{\end{lemm}}
\newcommand{\bedef}{\begin{defi}}
\newcommand{\bete}{\begin{teor}}
\newcommand{\eddef}{\end{defi}}
\newcommand{\ente}{\end{teor}}
\newcommand{\beos}{\begin{osse}}
\newcommand{\eddos}{\end{osse}}
\newcommand{\bepr}{\begin{prop}}
\newcommand{\empr}{\end{prop}}
\newcommand{\bepro}{\begin{prob}}
\newcommand{\empro}{\end{prob}}
\newcommand{\bede}{\begin{defin}}
\newcommand{\edde}{\end{defin}}
\newcommand{\beco}{\begin{coro}}
\newcommand{\enco}{\end{coro}}

\newcommand{\beeq}[1]{\begin{equation}
 \label{#1}}
\newcommand{\eddeq}{\end{equation}}

\newcommand{\beeqa}[1]{\begin{eqnarray}
  \label{#1}}
\newcommand{\eddeqa}{\end{eqnarray}}

\newcommand{\beal}[1]{\begin{align}
 \label{#1}}
\newcommand{\eddal}{\end{align}}

\newcommand{\bespl}[1]{\begin{split}
 \label{#1}}
\newcommand{\edspl}{\end{split}}

\newcommand{\bega}[1]{\begin{gather}
 \label{#1}}
\newcommand{\edga}{\end{gather}}

\newcommand{\beeqax}{\begin{eqnarray*}}
\newcommand{\eddeqax}{\end{eqnarray*}}

\newcommand{\no}{\nonumber}

\newcommand{\tensorev}[2]{\epsilon(\uu(#1,#2))}
\newcommand{\tensorevt}[2]{\epsilon(\mathbf{u}_t(#1,#2))}

\newcommand{\dt}{\partial_t}

\newcommand{\nn}{{\bf n}}
\newcommand{\uu}{{\bf u}}

\newcommand{\vv}{{\bf v}}

\newcommand{\eps}{\varepsilon}
\newcommand{\epsi}{\varepsilon}

\newcommand{\weak}{\rightharpoonup}
\newcommand{\weakstar}{\mathop{\rightharpoonup}^{*}}

\DeclareMathOperator{\dive}{div}

\let\TeXchi\chi
\def\chi{{\setbox0 \hbox{\mathsurround0pt
$\TeXchi$}\hbox{\raise\dp0 \copy0 }}}

\newtheorem{theorem}{Theorem}[section]

\newtheorem{lemma}{Lemma}[section]
\newtheorem{proposition}[lemma]{Proposition}
\newtheorem{definition}[lemma]{Definition}%
\newtheorem{remark}[lemma]{Remark}%

\newtheorem{problem}[lemma]{Problem}
\newtheorem{notation}[lemma]{Notation}

\renewcommand{\part}{\partial_t}

\newcommand{\debole}{\,\weak\,}
\newcommand{\debolestar}{\,\weakstar\,}
\newcommand{\pairing}[4]{ \sideset{_{#1 }}{_{ #2}}  {\mathop{\langle #3 , #4  \rangle}}}

 \def\fin{\hfill
         \trait .3 5 0
         \trait 5 .3 0
         \kern-5pt
         \trait 5 5 -4.7
         \trait 0.3 5 0
 \medskip}
 \def\trait #1 #2 #3 {\vrule width #1pt height #2pt depth #3pt}
\newcommand{\forae}{\text{for a.a.}}
\newcommand{\foraa}{\text{for a.a.}}
\newcommand{\aein}{\text{a.e.\ in}}

\newcommand{\down}{\downarrow}

\newcommand{\GC}{\Gamma_{\mathrm{C}}}
\newcommand{\overlineGC}{\overline{\Gamma}_{\mathrm{C}}}
\newcommand{\Neu}{{\mathrm{N}}}
\newcommand{\Dir}{{\mathrm{D}}}

\newcommand{\R}{\Bbb{R}}
\newcommand{\N}{\Bbb{N}}

\newcommand{\QED}{\mbox{}\hfill\rule{5pt}{5pt}\medskip\par}

\newcommand{\uuh}{\widehat{\uu}}
\newcommand{\chih}{\widehat{\chi}}

\newcommand{\xih}{\widehat{\xi}}
\newcommand{\zetah}{\widehat{\zzeta}}

\newcommand{\dd}{\,\mathrm{d}}

\newcommand{\gnl}{\mathbf{g}}

\newcommand{\rme}{\mathrm{e}}
\newcommand{\kr}{k}

\newcommand{\nlocs}[2]{\mathcal{K}[#1](#2)}
\newcommand{\nlocss}[1]{\mathcal{K}[#1]}
\newcommand{\nlname}{\mathcal{K}}
\newcommand{\bsY}{\mathbf{Y}}
\newcommand{\bsV}{\mathbf{V}}
\newcommand{\bsW}{\mathbf{W}}
\newcommand{\bsH}{\mathbf{H}}
\newcommand{\calR}{\mathcal{R}}

\newcommand{\calE}{\mathcal{E}}

\newcommand{\bvarphi}{\widehat{\balpha}}
\newcommand{\nl}{\mathrm{nl}}

\newcommand{\balpha}{{\mbox{\boldmath$\alpha$}}}
\newcommand{\zzeta}{{\mbox{\boldmath$\zeta$}}}
\newcommand{\bwalpha}{{\mbox{\boldmath$\widehat{\alpha}$}}}

\newcommand{\traccia}{{|_{\GC}}}
\newcommand{\rmC}{\mathrm{C}}
\newcommand{\Gdir}{\Gamma_{\mathrm{D}}}
\newcommand{\Gnew}{\Gamma_{\mathrm{N}}}
\newcommand{\overlineGdir}{\overline{\Gamma}_{\mathrm{D}}}
\newcommand{\overlineGnew}{\overline{\Gamma}_{\Neu}}
\newcommand{\piecewiseConstant}[2]{\overline{#1}_{\kern-1pt#2}}

\newcommand{\underlinepiecewiseConstant}[2]{\underline{#1}_{\kern-1pt#2}}

\newcommand{\piecewiseLinear}[2]{{#1}_{\kern-1pt#2}}

\newcommand{\pwM}[2]{\widetilde{#1}_{\kern-1pt#2}}
 \def\trait #1 #2 #3 {\vrule width #1pt height #2pt depth #3pt}
\newcommand{\pwN}[2]{#1_{\kern-1pt#2}}
 \def\trait #1 #2 #3 {\vrule width #1pt height #2pt depth #3pt}

 \newcommand{\sfT}{\mathsf{T}}

 \newcommand{\omegah}{\widehat\omega}

\begin{document}

\date{March 31, 2017}

                                %TITOLO
                                %
\title{\Large Global existence for  a  nonlocal model for adhesive contact}
                           %
                                %Autore/                                %

\author{Elena Bonetti}
\address{E.\ Bonetti, Dipartimento di Matematica ``F.\
Enriques'', Universit\`a di Milano. Via Saldini 50, I--20133 Milano -- ITALY}
\email{elena.bonetti\,@\,unimi.it}

\author{Giovanna Bonfanti}
\address{G.\ Bonfanti, Sezione di Matematica, DICATAM, Universit\`a degli studi di Brescia. Via Valotti 9, I--25133 Brescia -- ITALY}
\email{giovanna.bonfanti\,@\,unibs.it}

\author{Riccarda Rossi}
\address{R.\ Rossi, DIMI, Universit\`a degli studi di Brescia. Via Branze 38, I--25133 Brescia -- ITALY}
\email{riccarda.rossi\,@\,unibs.it}

\maketitle

 \numberwithin{equation}{section}

\begin{abstract}
In this paper we address the analytical investigation of a model for adhesive contact introduced in \cite{freddi-fremond}, which includes nonlocal sources of damage on the contact surface,
 such  as the elongation. The resulting PDE system features  various nonlinearities rendering the unilateral contact conditions, the physical constraints on the internal variables, as well as the %integral
contributions  related to the nonlocal forces.
 For the associated  initial-boundary value problem we obtain a \emph{global-in-time}
existence result by proving the existence of a local solution via a suitable approximation procedure  and then by  extending the local solution to a global
one by a 
nonstandard prolongation argument.
\end{abstract}
                                %
                                %headings e titolo breve in pagina
                                %
\pagestyle{myheadings} \markright{\it
Nonlocal  adhesion}
                                %
                                %TESTO
                                %

\section{Introduction}

The mathematical field of contact mechanics has  flourished  over the last decades, as illustrated for instance in the monographs \cite{eck05, sofonea-han-shillor, mon3}. In this paper, we focus on a
PDE system pertaining to a subclass of contact models, i.e.\  models for \emph{adhesive contact}. This phenomenon
%The phenomenon of \RRR adhesive contact   investigated in this paper
 plays  an important role in the  analysis  of the stability of structures. Indeed, it is well known that  interfacial  regions between materials are fundamental to ensure the strength and stability of structural elements.
We describe adhesive contact  by using a ``surface damage theory'', as  proposed by \textsc{M.\ Fr\'emond} (see e.g.\ \cite{fre}), for the action of the adhesive substance   (one may think of glue),  located on the contact surface.  This ``damage approach'' 
 provides an efficient and predictive theory for the mechanical behavior of the structure.
\par
 The main idea  underlying this modeling perspective is that, while  the basic unilateral contact theory does not allow for any resistance to tension, in adhesive contact resistance to tension is related to the action of microscopic bonds between the surfaces of the  adhering  solids.
 %\EEE  As a result,
  The resulting description of the state of the bonds between the two bodies is in the spirit of damage models. In fact, it is given
 % in terms of damage models, is useful to The description is given at a microscopic level introducing
 in terms of a phase parameter $\chi$  akin to  a damage parameter, characterizing the state of the cohesive bonds. The PDE system  is recovered  from  the balance laws of continuum mechanics, including (in a generalized principle of virtual power) the effects of micro-forces and micro-motions which are responsible for the breaking (or possibly repairing) of the  microscopic   bonds  on the contact surface.
 \par
 The mathematical analysis of models for adhesive contact and delamination \`a la Fr\'emond,
pioneered by  \cite{Point}, has attracted remarkable attention over the last 15 years, both in the case of \emph{rate-independent} (cf., e.g., \cite{MRK06,MRT12,RouSur}) and \emph{rate-dependent}
(see, among others, \cite{Raous},  \cite{sofonea2})
 evolution.
 \par
 The type of model  investigated in the present paper was first  rigorously derived and analyzed  in
 \cite{BBR1} (cf.\ also \cite{BBR2}). The associated PDE system couples an equation for the  macroscopic deformations of the body and a ``boundary'' equation on the contact surface, describing the evolution of the state of the glue in terms of a surface damage parameter. The system is highly nonlinear, mainly due to the presence of nonlinear boundary conditions and nonsmooth constraints on the internal variables,  providing the unilaterality of the contact,  the physical consistency of the damage parameter and, possibly, the unidirectional character of the degradation process).
  The model from \cite{BBR1} has been generalized  to the non-isothermal case,    with the temperature evolution  described by an entropy balance equation (cf.\ e.g.\ \cite{BBR3})  or by the more classical energy balance equation (see \cite{BBRen}).
  Furthermore,   in \cite{BBRfric}, \cite{BBRfrictemp}, and \cite{BBRen} friction effects have also been  included, even in the temperature-dependent  case. There,  we deal with the coupling between the Signorini condition for adhesive contact and a nonlocal  regularization of the Coulomb law %More precisely the regularization acts on the unbounded normal reaction as a coefficient of the tangential component.
 where the friction coefficient may depend on the temperature.
   \par
    In this paper we again focus the isothermal and frictionless case. In fact, we address the  analysis of a model  for adhesive contact that  generalizes that from \cite{BBR1}  by assuming  that also nonlocal forces act on the contact surface. This  leads to \ the presence of  novel  nonlocal (nonlinear) contributions in the  resulting  PDE system.
   More precisely, on the contact surface we consider 
 interactions 
 %\footnote{\RRR Ho tolto l'aggettivo 'local', qui, perch\'e mi sembrava che confondesse le idee, visto che sottolineiamo invece 'nonlocal'}
  between damage at a point and damage in its neighborhood
 %. Thus, we both prescribe  a \RRR nonlocal \EEE interaction within the \RRR adhesive substance \EEE
  (i.e.\ we use a gradient surface damage theory) and  moreover we admit  an interaction between the adhesive substance  and the two bodies. An example of this kind of behavior is given by experiments showing that  elongation, i.e.\ a variation of the distance of two distinct points on the contact surface,  may have damaging effects. Thus, a nonlocal quantity corresponding to the elongation is considered in the energy functionals.
  \par
  This model was  introduced in \cite{fre, freddi-fremond}.  While referring to the latter paper   for details  on the modeling   and comments  on the  applications and computational results,  in the following  lines we will outline the derivation of the model for the sake of completeness. %Like in \cite{freddi-fremond},
   We will   confine the discussion to
%We follow the approach introduced in \cite{freddi-fremond}, we refer to for further details.
 %ctually confining ourselves in
 the reduced case when only one body is considered in adhesive contact with a rigid support on a part of its boundary. We observe that this choice has the advantage of simplifying the exposition
 in comparison to the two-body case, while  affecting 
 neither  the relevance of the model nor its analytical investigation.
%%%%
%%%%
\subsection{The model and the PDE system}
Let us consider a body which is located  in a  sufficiently   smooth and bounded domain $\Omega\subset{\R}^3$ and lying 
on a rigid support during a finite time interval $(0,T)$. We denote its boundary  by
$\partial\Omega=\overlineGdir\cup\overlineGnew\cup\overlineGC$.
Here  $\Gdir, \, \Gnew,\, \GC$  are open subsets in the relative
topology of $\partial\Omega$, each of them with a smooth boundary
and disjoint one from each other.  In particular, $\GC$ is the
contact surface, which is considered as  a flat surface and identified  with a
subset of ${\R}^2$, cf.\ \eqref{hyp:GC} ahead.  Hereafter we shall suppose that both $\GC$ and
$\Gamma_\Dir$ have   positive measure.  We let ${\bf
u}={\bf 0}$ on $\Gdir$, while a known traction is
applied on $\Gnew$.

%We follow the modeling approach proposed by \cite{freddi-fremond}
%For the sake of simplicity, we do not include neither thermal evolution nor bulk damage processes nor friction effects.
 As already mentioned, in the model we neither encompass thermal evolution nor frictional effects. Thus, 
 the
state variables of the model  are
$$
(\epsilon(\uu),\chi,\nabla\chi, \uu_{|_{\GC}},  \gnl)
$$
where  we denote by $\uu$ the vector of small
displacements,  $\epsilon(\uu)$ the symmetric linearized strain tensor,  $\chi$ a damage parameter defined on
the contact
surface, $\nabla \chi$  its gradient, and
$\uu_{|_{\GC}}$ the trace of the displacement $\uu$ on $\GC$. The parameter $\chi$ is assumed to take values in
$[0,1]$, with $\chi=0$ for completely damaged bonds, $\chi=1$ for
undamaged bonds, and $\chi\in(0,1)$ for partially damaged bonds. 
%\GGG Moreover,
%in order to we introduce \EEE
The nonlocal term $\gnl$,
%\footnote{\RRR Ho rimpiazzato $g$ con $\gnl$ per coerenza con la notazione $\uu$, cmq ho usato dappertutto delle macro..} is
 defined  for $(x,y)\in \GC\times\GC$ by
\begin{equation}
	\label{def-gnl}
	 \gnl(x,y) = 2(x-y)\uu_{|_{\GC}}(x),
\end{equation}
 describes the damaging effects due to the elongation.
 The free energy of the system is given by the sum of a bulk, a local surface, and a nonlocal surface contributions. More precisely,
%\footnote{\RRR suggerisco di scrivere $ \Psi_{\nl} (\chi, \gnl (x,y))$, invece di  $\Psi_{\nl} (\chi(x),\chi(y), \gnl (x,y))$ come suggeriva Giovanna.. anche se la seconda scelta sarebbe piu' corretta, secondo me appesantirebbe troppo: dovremmo allora avere $\chi(x)$ e $\chi(y)$ anche nell'argomento di $\Psi$...}
%\footnote{\GGio A questo punto, dato che abbiamo tolto $\Psi$, io rimetterei $\Psi_{\nl} (\chi(x),\chi(y), \gnl (x,y))$. Infatti, nominare $x$ e $y$ e' coerente con $\gnl (x,y)$.}
%$$
%\Psi(\epsilon(\uu),\chi,\nabla\chi, \uu_{|_{\GC}}, \gnl (x,y))= \RRR \Psi_\Omega(\epsilon(\uu))+\Psi_{\GC}(\chi,\nabla\chi, \uu_{|_{\GC}})+  %\Psi_{\nl} (\chi, \gnl (x,y)) \EEE
%$$
%where
%we have a  bulk contribution
%$\Psi_\Omega$, \EEE and two surface ones, i.e. a local energy $\Psi_{\GC}$ and a nonlocal energy  $\Psi_{\nl}$.
the free energy  (density)  in $\Omega$ %(here we avoid the thermal part, which is assumed to be constant),
 is the classical one in elasticity theory
\begin{equation}\label{energiacorpo}
	\Psi_\Omega(\epsilon(\uu))=
	\frac 1 2 \epsilon(\uu)\mathbb{K}\epsilon(\uu),
\end{equation}
where   $\mathbb{K}=(a_{ijkh})$ stands for  the elasticity tensor.
 Moreover, the local surface part of the free energy (density) is given by 
\begin{equation}\label{energiabordo}
	\Psi_{\GC}(\chi,\nabla\chi, \uu_{|_{\GC}})=
	I_{[0,1]}(\chi)+\gamma(\chi)+\frac 1 2|\nabla\chi|^2+
	\frac 1 2 \chi|\uu_\traccia|^2+I_{(-\infty,0]}(\uu_\traccia\cdot{\bf n}),
\end{equation}
where  the indicator function $I_{[0,1]}$ of the
interval $[0,1]$ accounts for a  physical constraint on $\chi$, being
$I_{[0,1]}(\chi)=0$ if $\chi\in [0,1]$ and $I_{[0,1]}(\chi)=+\infty$
otherwise. Analogously, denoting by $I_{(-\infty,0]}$ the indicator function of
the interval $(-\infty,0]$, the term $I_{(-\infty,0]}(\uu_\traccia\cdot{\bf n})$
renders the impenetrability condition on the contact surface, as it
enforces that $\uu_\traccia\cdot{\bf n}\leq 0$ (${\bf n}$ is the
outward unit normal vector to $\GC$). Finally, the function
 $\gamma$,   sufficiently smooth and possibly nonconvex, is
related to non-monotone dynamics for $\chi$
(from a physical point of view, it
includes some cohesion in the material).
Finally, let us consider the nonlocal  surface  free energy (density), which is given by
%\footnote{\GGio ho riaggiunto $\chi(x), \chi(y)$ nell'argomento di $\Psi_\nl$ ..... \EEE}
%\footnote{\RRR $d$ \`e la dimensione dello spazio, che nel nostro caso \`e $d=3$...??}
\begin{equation}\label{energianonloc}
	 \Psi_{\nl}(\chi(x), \chi(y),\gnl(x,y)) 
	=\frac 1 2 \gnl^2(x,y)\chi(x)\chi(y)\rme^{\frac{-|x-y|^2}{d^2}},
\end{equation}
where the exponential function with distance $d$ describes the attenuation of nonlocal actions with distance $|x-y|$ between points $x$ and $y$ on the contact surface.
\par
As far as dissipation, we assume that
 there is no dissipation
  due to changes of 
  the nonlocal variable, while we  consider the dissipative variables $\epsilon(\uu_t)$ in $\Omega$ and $\chi_t$ in $\GC$. We
follow the
approach
proposed by \textsc{J.J. Moreau} to prescribe the
dissipated energy by means of the so-called pseudo-potential of
dissipation
which is a convex, nonnegative functional, attaining its minimal value $ 0$ 
when the  dissipation (described by the dissipative variables) is
zero.  More precisely,
we define the volume part $\Phi_\Omega$ of the pseudo-potential of
dissipation by
\begin{equation}\label{pseudocorpo}
	\Phi_\Omega ( \epsilon(\uu_t))=\frac 1 2 \epsilon(\uu_t)
	\mathbb{K}_v\epsilon(\uu_t),
\end{equation}
where $\mathbb{K}_v=(b_{ijkh})$ denotes the viscosity tensor. The surface part
$\Phi_{\GC}$ of the pseudo-potential of dissipation depends on
$\chi_t$
%and it is (here we have a rate-dependent evolution for $\chi$)
via
\begin{equation}\label{pseudobordo}
	\Phi_{\GC}(\chi_t)=\frac 1 2 |\chi_t|^2  + I_{(-\infty,0]}(\chi_t).
\end{equation}
The quadratic term $\frac 1 2 |\chi_t|^2$ encodes \emph{rate-dependent} evolution of $\chi$,
 where the indicator term $I_{(-\infty,0]}(\chi_t)$ forces $\chi_t$ to take nonpositive values and renders
the unidirectional  character of the damage process. 
\par
Hereafter, we shall omit for simplicity the index
$v_\traccia$ to denote the trace  on $\GC$ of a function $v$,
defined in $\Omega$. The equations are recovered by a generalization of the principle of virtual powers, in which microscopic forces (also nonlocal ones) responsible for the damage process in the adhesive substance are included. More precisely, for any virtual bulk velocity $\vv$ with $\vv= \bf{0}$ on $\Gdir$ and for any virtual microscopic velocity $w$  on the contact surface, we define the power of internal forces in $\Omega$ and $\GC$ as follows
\begin{equation}
\label{Pint}
\begin{aligned}
	{\mathcal P}_{int}&=-\int_\Omega\Sigma\epsilon(\vv) \dd \Omega -\int_{\GC} (B w+{\bf H}\nabla w) \dd x  +\int_{\GC}{\bf R}\vv \dd x \\
	&+\int_{\GC}\int_{\GC}2M(x,y)(x-y)\vv(x) \dd x \dd y
+\int_{\GC}\int_{\GC}(B^1_{\nl}(x,y) w(x)+B_{\nl}^2(x,y)w(y)) \dd x \dd y.
\end{aligned}
\end{equation}
Here, $\Sigma$ is the Cauchy stress tensor, ${\bf R}$ the classical macroscopic reaction on the contact surface, $B$ and ${\bf H}$ are local   interior forces, responsible
for the  degradation  of the  adhesive bonds between the body and the support.  The terms $M(x,y)$ and $B^i_{\nl}(x,y), i=1,2, $ are new scalar nonlocal
contributions: they stand for internal microscopic nonlocal forces on the contact surface and describe the effects of the elongation as a  source of damage.  The power of the external forces is given by 
\begin{equation}\label{Pext}
	{\mathcal P}_{ext}=\int_\Omega {\bf f}\vv \dd \Omega +\int_{\Gamma_\Neu}{\bf h}\vv \dd \Gamma,
\end{equation}
where
${\bf f}$ is a bulk known external force, while ${\bf h}$ is a given traction on $\Gnew$.  Note that, here we have disregarded external forces acting on the microscopic level and confined ourselves   to the case of null accelerations power.
\par
The  principle  of virtual powers, holding for every virtual microscopic and macroscopic velocities and every subdomain in $\Omega$,  leads to
the  quasistatic momentum balance
%(in the
%quasi-static case)\
%\footnote{\GGG ho tolto tutti i $\times (0,T)$ perche' in certe formule avremmo dovuto aggiungere la dipendenza esplicita da t e questo appesantiva... Il dominio spazio-tempo ricompare dal sistema \ref{PDE-INTRO}\EEE }
\begin{equation}\label{bilanciomacro}
	-\dive \Sigma={\bf f} \quad\text{in } \Omega,
\end{equation}
 which will be also posed in  a time interval $(0,T)$, and
 supplemented by the following boundary conditions 
\begin{align}
	&\label{reazione} \Sigma{\bf n}(x)={\bf R}(x)+\int_{\GC} 2(x-y)M(x,y)  \dd y  \quad\text{in }\GC,\\
	&{\uu}={\bf 0}\quad\text{in }\Gamma_\Dir,\quad\Sigma{\bf
		n}={\bf h}\quad\text{in }\Gamma_\Neu.
\end{align}
Observe that in \eqref{reazione}
 the boundary condition for the stress tensor on the contact surface
combines a local contribution involving the (pointwise) reaction ${\bf R}(x)$ and a  nonlocal force (defined in terms of the new variable $M(x,y)$), related to  the  elongation.  %This is the main novelty of our contribution with respect to the previous literature. It results

Again, the principle of virtual powers leads to 
a micro-force balance on the contact surface,  also posed in the time interval $(0,T)$,  given by
\begin{equation}\label{bilanciomicro}
	 B(x)-\dive {\bf H} (x)= \int_{\GC} \left( B^1_{\nl}(x,y)+B_{\nl}^2(y,x)\right) \dd y  \ \text{ in }\GC,\qquad{\bf H}\cdot{\bf n}_s=0
	\,\text{ on }\partial\GC.
\end{equation}
Here, ${\bf n}_s$ denotes the outward unit normal vector to
$\partial\GC$.

Constitutive relations for $\Sigma, {\bf R}, B, {\bf
	H}, M$, and $B^i_{\nl}, i=1,2,$ are given in terms of the free energies and the
pseudo-potentials of dissipation. More precisely, the constitutive relation for the stress tensor $\Sigma$ accounts
for the dissipative (viscous) dynamics of the deformations
\begin{equation}
	\Sigma=\frac{\partial\Psi_\Omega}{\partial\epsilon(\uu)}+
	\frac{\partial\Phi_\Omega}{\partial\epsilon(\uu_t)}=\mathbb{K}\epsilon(\uu) +
	\mathbb{K}_v\epsilon(\uu_t),
\end{equation}
while the local reaction ${\bf
	R}$ is given by
\begin{equation}
	{\bf
		R}=-\frac{\partial\Psi_{\GC}}{\partial\uu}\in -\chi\uu
	-\partial I_{(-\infty,0]}(\uu\cdot {\bf n}){\bf n}.
\end{equation}
As for as the nonlocal force on $\GC$, we  prescribe
\begin{equation}
	M(x,y)=-\frac{\partial\Psi_{\nl}}{\partial \gnl(x,y)}=-\gnl(x,y)\chi(x)\chi(y)\rme^{-\frac{|x-y|^2}{d^2}},
\end{equation}
%
%\footnote{\RRR Dobbiamo spiegare un po' di piu' questa formula: anche
%	se \`e intesa formalmente (e serve per fare tornare i conti...), che senso ha la derivata%
%	\[
%	\left(\frac{\partial\Psi_{nl}}{\partial\chi(x)}+
%	\frac{\partial\Psi_{nl}}{\partial\chi(y)}\right)??
%	\]
%	\ELE Ho messo qualche dettaglio in pi' se va meglio...\EEE}
while
the terms $B_{\nl}^1$ and $B_{\nl}^2$ are  (formally) defined as derivatives of $\Psi_{\nl}$  with respect to the values of the surface  damage parameter in $x$ and $y\in \GC$, respectively,  as it follows
\begin{align}
	& B^1_{\nl}(x,y)=-\frac{\partial\Psi_{\nl}}{\partial\chi(x)}=-\frac {\gnl^2(x,y)}2\rme^{-\frac{|x-y|^2}{d^2}}\chi(y),\\
	& B^2_{\nl}(x,y)=-\frac{\partial\Psi_{\nl}}{\partial\chi(y)}=-\frac {\gnl^2(x,y)}2\rme^{-\frac{|x-y|^2}{d^2}}\chi(x).
\end{align}
We further prescribe $B$ by
\begin{equation}
	B=\frac{\partial\Psi_{\GC}}{\partial\chi}+\frac{\partial\Phi_{\GC}}{\partial\chi_t}\in\partial
	I_{[0,1]}(\chi)+\gamma'(\chi) +\frac 1 2 |\uu|^2+\chi_t  + \partial I_{(-\infty,0]}(\chi_t),
\end{equation}
($ \partial I_{(-\infty,0]}$ and $\partial I_{[0,1]}$  denoting the convex analysis 
subdifferentials of $I_{(-\infty,0]}$ and $I_{[0,1]}$, resp.), and let
${\bf H}$ be
\begin{equation}
	{\bf
		H}=\frac{\partial\Psi_{\GC}}{\partial\nabla\chi}=\nabla\chi.
\end{equation}

Combining the previous constitutive relations with the balance laws,
we obtain
the following  boundary value problem
\begin{subequations}
	\label{PDE-INTRO}
	\begin{align}
		\label{eqI}
		&-\hbox{div }(\mathbb{K}\tensorev xt+ \mathbb{K}_v\tensorevt xt )={\bf f}(x,t), && (x,t) \in \Omega\times (0,T),\\
		\label{Dir}
		&\mathbf{u}(x,t)= {\bf 0}, && (x,t)\in \Gamma_{\mathrm{D}}\times (0,T),
		\\
		&
		(\mathbb{K}\tensorev xt+\mathbb{K}_v\tensorevt xt ){\bf n}={\bf h}(x,t), && (x,t)\in \Gamma_{\mathrm{N}}\times (0,T),\label{condIi}\\
		& (\mathbb{K}\tensorev xt+ \mathbb{K}_v\tensorevt xt  ){\bf n}+\chi(x,t){\bf u}(x,t)+
		\partial I_{(-\infty,0]}(\uu(x,t)\cdot {\bf n}){\bf n} && \no\\
		&\quad +\int_{\GC}2(x-y)\gnl(x,y)\chi(x,t)\chi(y,t)\rme^{-\frac{|x-y|^2}{d^2}}
		\dd y
		\ni{\bf 0}, && (x,t) \in  \GC \times (0,T),\label{condIii}
		\\
		&\chi_t (x,t)+  \partial I_{(-\infty,0]}(\chi_t(x,t)) -\Delta\chi(x,t)+\partial I_{[0,1]}(\chi(x,t)) +
		\gamma'(\chi(x,t))\no\\
		& \ni -\frac 1 2\vert{\uu (x,t)}\vert^2-\frac 1 2\int_{\GC}\left(\gnl^2(x,y)+\gnl^2(y,x)\right)\chi(y,t)\rme^{-\frac{|x-y|^2}{d^2}}
		\dd y, && (x,t) \in\GC
		\times (0,T),\label{eqII}
		\\
		&\partial_{\nn_s} \chi(x,t)=0  &&  (x,t)\in\partial \GC \times
		(0,T),\label{bord-chi}
	\end{align}
\end{subequations}
 where all integrals on $\GC$ involve the Lebesgue measure, which coincides with the Hausdorff measure on $\GC$ by the flatness requirement, cf.\ \eqref{hyp:GC} ahead.
% With slight abuse of notation, we have denoted both  the  \EEE
% (which will in fact coincide with the Lebesgue measure as we will suppose $\GC$ to be a \emph{flat}\footnote{\RRR questo dovrebbe servire solo se teniamo il termine $\Delta \chi$.. se ne facciamo a meno, allora non dobbiamo aggiungere questa ipotesi su $\GC$, e possiamo tenercelo generale (usando la misura di Hausdorff..)} surface). \EEE
%%
 Let us stress once again, with respect to the `standard' Fr\'emond's system for adhesive contact,  \eqref{PDE-INTRO}
encompasses  nonlocal terms both in the normal reaction on the contact surface (cf. with \eqref{condIii}) and in the flow rule \eqref{eqII} for $\chi$. In particular, we note that the right-hand side of \eqref{eqII} (see also the second of \eqref{nucleo}) may be different from zero even if $\uu(x,t)={\bf 0}$, due to the integral contributions which render the damaging effects of the elongation.
\par
 Taking into account the explicit expression \eqref{def-gnl} of $\gnl$,  observe  the integral terms on the left-hand side of \eqref{condIii}  and on the right-hand side of
\eqref{eqII}
can be  rewritten  as
\begin{align}
\label{nucleo}
\begin{cases}
\chi(x,t)\uu(x,t) \int_{\GC} \eta(x-y) \chi(y,t) \dd y,
\\
-\frac12 |\uu|^2(x,t) \int_{\GC} \eta(x-y) \chi(y,t) \dd y -\frac12 \int_{\GC} \eta(x-y) \chi(y,t) |\uu|^2(y,t)\dd y
\end{cases}
\quad \text{with } \eta(\zeta) = 4\zeta^2 \rm\rme^{-\zeta^2/d^2}\,.
\end{align}
%\footnote{\GGG questa parte la metterei tra le varie generalizzazioni del sistema \ref{PDE-INTRO} cioe' nella prossima sezione \RRR Io invece la terrei qui.. mi sembra che si leghi bene al discorso in questo punto, e poi forse vogliamo sottolineare questa generalizzazione, riepetto alle altre, che ormai sono 'standard'..}  
  In view of  \eqref{nucleo}, it is natural to address the analysis of a \emph{generalization} of system
\eqref{PDE-INTRO},  cf.\ \eqref{PDE} ahead,
where the kernel $\eta(x-y)$, featuring the bounded, even, and positive function $\eta \in \mathrm{C}^0(\R)$, is replaced by a
bounded, symmetric, and positive  kernel $\kr: \GC\times \GC \to [0,\infty)$, inducing the \emph{nonlocal} operator
\begin{equation}
	\label{nl-op-intro}
	\nlocs{w}x: = \int_{\GC} \kr(x,y) w(y) \dd y, \ \ w \in L^1(\GC).
\end{equation}
\subsection{Analytical results and plan of the paper}
 Our main result, \underline{\bf Theorem \ref{th:main}}, states the existence of global-in-time solutions for the Cauchy problem associated with a \emph{generalized} version of system
\eqref{PDE-INTRO}, featuring the nonlocal operator $\nlname$, and where the various, concrete, subdifferential operators  are replaced by maximal monotone nonlinearities.
It is stated in \underline{Section 2}, where we collect all the assumptions on the problem data and introduce a suitable variational formulation of our PDE system.
 In \underline{Section 3} we set up a suitable approximation, for which   we prove a local-in-time well-posedness  result by means of Schauder fixed point tecnique. % for the local existence combined with contraction estimates for the continuous dependence.
In \underline{Section 4}, by a priori estimates combined with compactness and monotonicity tools, we  develop a passage to the limit argument and we obtain a local-in-time solution for the original problem. Finally, in \underline{Section 5} we complete the proof of Theorem
 \ref{th:main}, extending the local solution to a global one by a carefully  devised  prolongation procedure.
%%%%
%%%%
\section{Setup and  main result}
\label{s:2}
\noindent
 Throughout the paper we shall assume that
 \begin{equation}
 \label{hyp:GC}
 \begin{gathered}
 \text{
 $\Omega$ is a
bounded   Lipschitz  domain in $\R^3$, with }
\\
\partial\Omega= \overlineGdir \cup \overlineGnew \cup \overlineGC, \ \text{ $\Gdir$, $\Gnew$, $\GC$,
 open disjoint subsets in the relative topology of $\partial\Omega$,
such that  }
\\
\mathcal{H}^{2}(\Gdir), \,  \mathcal{H}^{2}(\GC)>0, \qquad
\text{and ${\GC} \subset \R^2$ a  sufficiently smooth
\emph{flat} surface.}
\end{gathered}
\end{equation}
 More precisely,
by \emph{flat} we mean that $\GC$ is a subset of a hyperplane of
$\R^3$ and on $\GC$  the Lebesgue and Hausdorff measures $\mathcal{L}^2 $ and $
\mathcal{H}^2$ coincide.
 As for smoothness, we require that
$\GC$ has a   $\mathrm{C}^{1,1}$-boundary.
%\footnote{\RRR DUBBIO: ha senso chiedere $\Omega$ solo Lipschitz, e poi $\GC$ con bordo liscio?? Abbiamo fatto cosi' in BBR8, ma ora ho dei dubbi.. Per questa cosa ho scritto a Dorothee}
%In what follows, we will omit to recall assumption \eqref{hyp:GC} in the statements of the various results.
\par
 Hereafter,  will use the following
\begin{notation}
\label{notation-1} \upshape
 Given a Banach space $X$, we denote by
 $\pairing{}{X}{\cdot}{\cdot}$ the duality pairing
between its dual space  $X'$ and $X$ itself  and  by $\Vert\cdot\Vert_X$
both the norm in $X$ and in any power of it. In particular, we shall
use  short-hand notation for some function spaces
%\footnote{\RRR ho un po' cambiato la notazione per gli spazi funzionali, introducendo lo spazio $W$ per la $\chi$. E allora, per coerenza, ho rimpiazzato $\bsW$ con $\bsV$..}
\[
\begin{gathered}
 \bsH: =L^2 (\Omega;\R^3),
 %\quad
% \bsV =H^1 (\Omega;\R^3),
  \quad \bsV:=\{{\bf v}\in H^1 (\Omega;\R^3)\, : \ {\bf v}={\bf 0}\hbox{ a.e. on
}\Gamma_\Dir \}, \quad
 \bsY:
= H^{1/2}_{00,\Gamma_\Dir}(\GC;\R^3),
\\
H := L^2 (\GC), \quad V := H^1(\GC), \quad W: = \{ \chi \in H^2(\GC)\, : \ \partial_{\mathbf{n}_s}\chi=0 \},
\end{gathered}
\]
where we recall that
\begin{equation}
\label{def-space}
 H^{1/2}_{00,\Gamma_\Dir}(\GC;\R^3)=
\Big\{ \mathbf{w} \in H^{1/2}(\GC;\R^3)\, :  \ \exists\,
\tilde{\mathbf{w}}\in H^{1/2}(\Gamma;\R^3) \text{ with }
\tilde{\mathbf{w}}=\mathbf{w}  \text{ in $\GC$,} \
\tilde{\mathbf{w}}=\mathbf{0} \text{ in $\Gamma_\Dir$} \Big\}\,.
\end{equation}
The space $\bsV$  is endowed with the natural norm induced by
$H^1 (\Omega;\R^3)$.
Throughout the paper, we will also use that
\begin{equation}
\label{embedding}
\bsV \subset L^4(\GC) \text{ continuously, and } \bsV \Subset L^{4-s}(\GC) \text{ compactly for all  }  s \in (0,3],
\end{equation}
where the above embeddings have to be understood in the sense of traces.
%\RRR NOTATION FOR CONSTANTS.. \EEE
\par
Finally, we will use the symbols $c, c', C, C',\ldots$, with meaning possibly varying in the same line,  to denote several positive constants only depending on known quantities. Analogously, with the symbols $I_1,I_2,\ldots$ we will denote several integral terms appearing in the estimates.
\end{notation}
%%%%%%%%%%%%%%%%%%%%%%%%%%%%
\paragraph{\bf The bilinear forms in the momentum balance.} % elasticity theory.}
%%%%%%%%%%%%%%%%%
We recall the definition of  the standard bilinear forms of linear
viscoelasticity, which are involved in the variational formulation
of equation~\eqref{eqI}. Dealing with an anisotropic and
inhomogeneous material,
 we assume that the fourth-order tensors $\mathbb{K}=(a_{ijkh})$ and  $\mathbb{K}_v=(b_{ijkh})$,
denoting the elasticity and the viscosity tensor, respectively,
satisfy the classical symmetry and ellipticity conditions
\begin{subequations}
\label{ass-K}
\begin{equation}
\label{ass-K-sym-ell}
\begin{aligned}
& a_{ijkh}=a_{jikh}=a_{khij}\,,\quad  \quad
b_{ijkh}=b_{jikh}=b_{khij}\,,
  \quad   i,j,k,h=1,2,3,
\\
&  \exists \, \alpha_0>0 \,:  \qquad a_{ijkh} \xi_{ij}\xi_{kh}\geq
\alpha_0\xi_{ij}\xi_{ij} \quad   \forall\, \xi_{ij}\colon \xi_{ij}=
\xi_{ji}\,,\quad i,j=1,2,3\,,
\\
& \exists \, \beta_0>0 \,:  \qquad b_{ijkh} \xi_{ij}\xi_{kh}\geq
\beta_0\xi_{ij}\xi_{ij} \quad   \forall\, \xi_{ij}\colon \xi_{ij}=
\xi_{ji}\,,\quad i,j=1,2,3\,,
\end{aligned}
\end{equation}
where the usual summation convention is used. Moreover, we require
\begin{equation}
\label{ass-K-bdd}
a_{ijkh}, b_{ijkh} \in L^{\infty}(\Omega)\,, \quad  i,j,k,h=1,2,3.
\end{equation}
\end{subequations}
By the previous assumptions on the elasticity and viscosity
coefficients, the following bilinear forms $a, b : H^1 (\Omega;\R^3) \times H^1 (\Omega;\R^3)
\to \R$,   defined by
%\footnote{\RRR propongo di denotare lo small-strain tensor con il simbolo $\epsilon$ per non confonderci con il parametro di approssimazione $\varepsilon$..}
$$
\begin{aligned}
&
a({\bf u},{\bf v}):=\int_{\Omega} a_{ijkh} \epsilon_{kh}({\bf u})
\epsilon_{ij}({\bf v}) \dd x   &&   \text{for all } \uu, \vv \in H^1 (\Omega;\R^3),
\\
&
b({\bf u},{\bf v}):=\int_{\Omega} b_{ijkh} \epsilon_{kh}({\bf u})
\epsilon_{ij}({\bf v}) \dd x   &&   \text{for all }  \uu, \vv \in H^1 (\Omega;\R^3)
\end{aligned}
$$
turn out to be continuous and symmetric. In particular, we have
\begin{equation}
\label{continuity} \exists \, M >0: \ |a(\uu, \vv)| +   |b(\uu,
\vv)| \leq M \| \uu\|_{H^1 (\Omega)} \| \vv\|_{H^1 (\Omega)} \quad \text{for all }
\uu, \vv \in H^1 (\Omega;\R^3).
\end{equation}
%Moreover, note that \COMMENT{Se non usiamo mai questa formula, forse
%potremmo anche toglierla}
%$$
%b(\uu,\uu)=\|\varepsilon(\uu)\|_{H}^2\qquad \forall\, \uu \in
%\bfw\,.
%$$
Moreover, since $\Gamma_\Dir$ has positive measure,
 by Korn's inequality we deduce that $a(\cdot,\cdot)$ and
$b(\cdot,\cdot)$ are $\bsV$-elliptic, i.e., there exist $C_{a},
C_{b}>0 $ such that
\begin{align}
\label{korn}
 a({\bf u},{\bf u})\geq C_a\Vert{\bf u}\Vert^2_{\bsV},
 \qquad b({\bf u},{\bf u})\geq C_b\Vert{\bf u}\Vert^2_{\bsV} \qquad
\text{for all }\uu\in \bsV.
\end{align}
%%%%%
\paragraph{\bf Assumptions on the nonlinearities of the system.}
We will consider an extended version of system \eqref{PDE-INTRO},  where the subdifferentials $\partial I_{(-\infty,0]}$ in the  boundary condition \eqref{condIii} and in the flow rule \eqref{eqII} for $\chi$,
and $\partial I_{[0,1]}$ in \eqref{eqII},   are replaced by more general subdifferential operators.
\begin{enumerate}
\item We consider  a function
\begin{equation}
\label{hyp-alpha}
\widehat{\alpha} : \R \to [0,+\infty] \quad \text{proper, convex, and lower semicontinuous, with } \widehat{\alpha}(0)=0.
\end{equation}
 Note that  as soon as $0\in \mathrm{dom}(\widehat\alpha)$, we can always reduce to the case
$\widehat{\alpha}(0)=0$ by a translation.
 Then, we  introduce the proper, convex and lower semicontinuous functional
\[
\bwalpha:  \bsY \to [0,+\infty] \quad \text{ defined by } \quad \bwalpha(\uu): = \begin{cases}
\int_{\GC}\widehat{\alpha}(\uu \cdot \mathbf{n}) \,\dd x & \text{if } \widehat{\alpha}(\uu\cdot\mathbf{n}) \in L^1(\GC),
\\
+\infty &\text{otherwise}.
\end{cases}
\]
We set
$
\balpha: = \partial \bwalpha: \bsY\rightrightarrows \bsY'.
$
It follows from \eqref{hyp-alpha} that $\mathbf{0} \in \balpha(\mathbf{0})$.
\item
We consider
\begin{equation}
\label{hyp-rho}
\widehat{\rho} : \R \to [0,+\infty] \quad \text{proper, convex, and lower semicontinuous, with }
\mathrm{dom}(\widehat{\rho})\subset (-\infty,0] \text{ and } \widehat{\rho}(0)=0,
\end{equation}
and set $\rho: = \partial\widehat{\rho}:\R \rightrightarrows \R$.
\item
We let
 \begin{equation}
\label{hyp-beta}
\widehat{\beta} : \R \to [0,+\infty] \quad \text{proper, convex, and lower semicontinuous, with }
\mathrm{dom}(\widehat{\beta})\subset [0,+\infty) \text{ and } \widehat{\beta}(0)=0,
\end{equation}
and set $\beta: = \partial\widehat{\beta}:\R \rightrightarrows \R$.
\end{enumerate}
 The operator
$\balpha:  \bsY\rightrightarrows \bsY'$ will replace $\partial I_{(-\infty,0]}$ in the boundary condition \eqref{condIii}. Thus, as soon as $\mathrm{dom}(\widehat\alpha)\subset (-\infty,0]$, with $\balpha$ we render the impenetrability  (unilateral)  constraint $\uu\cdot\mathbf{n}\leq 0$
a.e.\ on  $\GC\times (0,T)$.
 The operators $\rho$ and
$\beta$ will generalize the subdifferentials
$\partial I_{(-\infty,0]}$
 and
 $\partial I_{[0,1]}$ in \eqref{eqII}. 
  On the one hand, the requirement
  % $\mathrm{dom}(\widehat{\rho})\subset (-\infty,0]$ and
 $\mathrm{dom}(\widehat{\beta})\subset [0,+\infty) $
  guarantees
 % $\chi_t\leq0$ and
 $\chi\geq0$ a.e.\ in $ \GC\times (0,T)$. On the other hand, starting from an initial datum $\chi_0 $ fulfilling $0\leq \chi_0\leq 1$ a.e.\ on $\GC$
  (cf.\ \eqref{ini-chi} below) and
  taking into account that $\chi_t\leq0$  a.e.\ in $ \GC\times (0,T)$ since $\mathrm{dom}(\widehat{\rho})\subset (-\infty,0] $, 
  we will ultimately deduce that
  $\chi\in [0,1]$ a.e.\ on $\GC \times (0,T)$.
\begin{enumerate}
 \setcounter{enumi}{3}
 \item
 As for the kernel $\kr$ defining the operator $\nlname $ from \eqref{nl-op-intro}, we will require that
 \begin{equation}
 \label{hyp-k}
 \kr:\GC \times \GC \to [0,+\infty) \text{ is symmetric, with } \kr \in L^\infty(\GC{\times}\GC)\,.
 \end{equation}
 \item Finally, we will suppose that
 %\footnote{\RRR da uniformare con l'introduzione: Gio' e io abbiamo chiamato $\gamma'$ il termine non-monotono nell'equazione per la $\chi$.... quindi nell'energia compare $\gamma$, per cui chiediamo la \eqref{hyp-gamma}...}
 \begin{equation}
 \label{hyp-gamma}
 \gamma \in \mathrm{C}^2(\R) \quad \text{with } \gamma' \text{ Lipschitz on } \R.
 \end{equation}
 \end{enumerate}
 \paragraph{\bf Assumptions on the problem data.}
We suppose that
\begin{subequations}
\label{hyp:initial}
\begin{align}
\label{ini-u}
 &
 \uu_0 \in \bsV \ \text{with} \ \uu_{0} \in  \mathrm{dom} (\bvarphi)\,,
 \\
&
 \label{ini-chi}
\begin{aligned}
 \chi_0\in W, \quad 
  \chi_0 \in [0,1]   \text{  on } \GC, \quad  \beta^0 (\chi_0) \in H, 
\end{aligned}
\end{align}
\end{subequations}
where $\beta^0(\chi_0)$ denotes the minimal section of
$\beta(\chi_0)$.
Then, from
\[
0 \leq \widehat{\beta}(\chi_0) \leq \beta^0(\chi_0) \chi_0 \quad \text{ a.e.\ on }\GC
\]
(which follows from the positivity of $\widehat\beta$
and from the fact that $\widehat\beta(0)=0$), we immediately deduce that
\[
\widehat{\beta}(\chi_0) \in L^1(\GC).
\]
 As far as the body force $\mathbf{f} $ and the
surface traction $\mathbf{h} $  are concerned,
 we prescribe that
 \begin{subequations}
\label{hyp:forces}
\begin{align}
& \label{hyp:f} \mathbf{f} \in L^2 (0,T;\bsH),
\\
& \label{hyp:h}  \mathbf{h} \in L^2 (0,T;
(H_{00,\Gamma_\Dir}^{1/2}(\Gamma_\Neu;\R^3))'),
\end{align}
\end{subequations}
and we introduce $\mathbf{F}:(0,T) \to \bsV'$ by
\begin{equation}
\label{effegrande} \pairing{}{\bsV}{\mathbf{F}(t)}{\vv}:=
\int_\Omega \mathbf{f}(t) \cdot \vv \dd x 
+\pairing{}{H_{00,\Gamma_\Dir}^{1/2}(\Gamma_\Neu;\R^3)}{\mathbf{h}(t)}{\vv} \quad \text{for all
} {\vv} \in\bsV \quad \forae \, t \in (0,T).
\end{equation}
Of course, thanks to \eqref{hyp:forces},  $\mathbf{F} \in L^2
(0,T;\bsV')$.
%%%%%%
\par
We are now in a position to give the variational formulation of the
initial-boundary value problem for the (generalized) version of system \eqref{PDE-INTRO}
tackled in this paper.
\begin{problem}
\label{prob:PDE}
Starting from initial data $(\uu_0,\chi_0)$ fulfilling \eqref{hyp:initial},
find a quintuple
%\footnote{\RRR problema di LaTeX: vorrei scrivere $\zeta$ in bold ma non ci riesco...}
 $(\uu,\zzeta,\chi,\omega,\xi)$ with
\begin{subequations}
\label{reg}
\begin{align}
&
\label{reg-u}
\uu\in H^1(0,T;\bsV),
\\
&
\label{reg-zeta}
\zzeta\in L^2(0,T;\bsY'),
\\
&
\label{reg-chi}
\chi\in L^2(0,T;W) \cap L^\infty (0,T;V)  \cap H^1(0,T;H),
\\
&
\label{reg-omega/xi}
\omega,\, \xi\in L^2(0,T;H)
\end{align}
\end{subequations}
such that $\uu(0)=\uu_0$, $\chi(0)=\chi_0$,
and fulfilling a.e.\ in (0,T)
\begin{subequations}
\label{PDE}
\begin{align}
&  \label{eqU}
\! \!\!\!\!\!\!
b(\uu_t,\mathbf{v}) + a(\uu,\mathbf{v}) +
 \int_{\GC} \chi\uu \mathbf{v} \dd x +
 \langle \zzeta,\mathbf{v} \rangle_{\bsY}
+  \int_{\GC} \chi\uu \nlocss\chi \mathbf{v} \dd x =
 \pairing{}{\bsV}{\mathbf{F}}{\vv} && \text{for all } \vv \in \bsV,
 \\
 &
 \label{zeta}
 \! \!\!\!\!\!
% \! \!\!\!\!\!\!\!\!\!
% \! \!\!\!\!\!\!\!\!\!\!\!
 \zzeta \in \balpha(\uu) && \text{in } \bsY',
 \\
 &
 \label{eqCHI}
 \! \!\!\!\!\!\!
% \! \!\!\!\!\!\!\!\!\!\!\!
 \chi_t + \omega -\Delta\chi+ \xi + \gamma'(\chi) = -\frac12 |\uu|^2 - \frac12 |\uu|^2 \nlocss\chi  -\frac12 \nlocss{\chi|\uu|^2} &&  \text{ a.e.\ on } \GC,
 \\
 &
 \label{omega}
 % \! \!\!\!\!\!\!\!\!\!\!\!
 \! \!\!\!\!\!
 \omega \in \rho(\chi_t) &&   \text{ a.e.\ on } \GC,
  \\
 &
 \label{xi}
  %\! \!\!\!\!\!\!\!\!\!\!\!
 \! \!\!\!\!\!\! \xi \in \beta(\chi) &&   \text{ a.e.\ on } \GC.
 \end{align}
\end{subequations}
\end{problem}
\noindent

 Now we state the main result of this paper, which ensures the existence of solutions to Problem \ref{prob:PDE}.
Let us mention in advance that, for the selections $ \omega \in \rho(\chi_t) $ and $\xi \in \beta(\chi)$ we will obtain stronger integrability properties
than those required with \eqref{reg-omega/xi}, cf.\ \eqref{further-reg} below. They are obtained by means of  a special a priori estimate, first introduced in
\cite{bfl}, which is tailored to the doubly nonlinear structure of  \eqref{eqCHI}  and has to be performed in order to separately estimate the selections in $\rho(\chi_t)$ and $\beta(\chi)$, cf.\
also Remark \ref{rmk:features-reg-prob} ahead. It is in view of this estimate that the conditions $\chi_0 \in  W$ 
(in accordance with the property $\chi \in L^\infty(0,T;W) \cap W^{1,\infty}(0,T;H) \subset  \mathrm{C}_{\mathrm{weak}}^0([0,T];W)$) 
 and $\beta^0(\chi_0) \in H$
from \eqref{ini-chi} are needed. 
\begin{theorem}
\label{th:main}
Assume \eqref{hyp:GC}, \eqref{ass-K}, \eqref{hyp-alpha}--\eqref{hyp-gamma},
and \eqref{hyp:forces}. Then, for any pair $(\uu_0,\chi_0)$ of initial data fulfilling
\eqref{hyp:initial}
there exists a quintuple $(\uu,\zzeta,\chi,\omega,\xi)$ solving  Problem \ref{prob:PDE}
 and
 \begin{enumerate}
 \item
 enjoying   the enhanced
 integrability properties
%such that, in addition,
\begin{equation}
\label{further-reg}
\chi\in L^\infty (0,T;W)    \cap H^1(0,T;V)   \cap  W^{1,\infty}(0,T;H)   \text{ and } \omega, \, \xi \in L^\infty (0,T;H),
\end{equation} and  such that 
 \begin{equation}
 \label{confinement}
  0 \leq \chi(x,t) \leq 1  \quad  \text{for all } (x,t)\in  \overlineGC\times [0,T],
 \end{equation}
\item
fulfilling the energy-dissipation  inequality 
\begin{equation}
\label{enid-balance}
 \int_s^t 2\calR(\uu_t(r),\chi_t(r)) \dd r +\int_s^t  \int_{\GC} \widehat{\rho}(\chi_t) \dd x + \calE(\uu(t),\chi(t)) \leq  \calE(\uu(s),\chi(s))   +\int_s^t
\pairing{}{\bsV}{\mathbf{F}(r)}{\uu_t(r)} \dd r
\end{equation}
for all $ 0 \leq s \leq t \leq T$,  with $\calR$ and $\calE$ given by \eqref{underlying-diss} and \eqref{underlying-en} below.
\par
In fact, \eqref{enid-balance} holds as a \emph{balance} if $\widehat{\rho}$ is positively  homogeneous of degree $1$ (i.e., $\widehat{\rho}(\lambda v) = \lambda \widehat{\rho}(v)$ for all $\lambda\geq 0$), e.g.\ in the particular case $\widehat{\rho} = I_{(-\infty,0]}$.
 \end{enumerate}
\end{theorem}
\noindent
\begin{remark}
\upshape The energy-dissipation  inequality   \eqref{enid-balance} in fact holds along \emph{any} solution to Problem \ref{prob:PDE}. It reflects the fact that system \eqref{PDE} has a \emph{(generalized) gradient system} structure.  In fact, it can be rewritten as
the (abstract) doubly nonlinear differential inclusion
\begin{equation}
\label{gen-gra-sys}
\partial \calR(\uu_t(t),\chi_t(t)) + \partial^- \calE(\uu(t),\chi(t)) \ni 0 \qquad \text{in } \bsV' \times H \qquad \foraa\, t \in (0,T),
\end{equation}
 involving the
dissipation potential
$\calR:  \bsV \times H \to [0,+\infty]$
\begin{equation}
\label{underlying-diss}
\calR(\uu_t,\chi_t):= \mathcal{R}_{\mathsf{u}}(\uu_t) + \calR_{\mathsf{\chi}}(\chi_t), \qquad
\text{with }
\begin{cases}
 \mathcal{R}_{\mathsf{u}}(\uu_t): =  \frac12 b(\uu_t,\uu_t),
\\
 \calR_{\mathsf{\chi}}(\chi_t): = \frac12 \|\chi_t\|_{H}^2 ,
 \end{cases}
\end{equation}
and the energy functional
%\footnote{\RRR qui potremmo richiamare la notazione per le densita' di energia (che abbiamo usato nella derivazione del modello) per i vari contributi a $\calE$.. in modo da legare questa parte a quello che \`e stato scritto nell'introduzione}
\begin{equation}
\label{underlying-en}
\begin{aligned}
 \calE(\uu,\chi): =  &   \mathcal{E}_1(\uu)   + \mathcal{E}_2(\uu,\chi)
\qquad \text{with }
\\
&
 \mathcal{E}_1(\uu) =
 \frac12 a(\uu,\uu) + \bwalpha(\uu),  %- \pairing{}{\bsV}{\mathbf{F}(t)}{\uu},
 \\
 & \mathcal{E}_2(\uu,\chi): =  \frac12 \int_{\GC} \chi |\uu|^2 \dd x +\frac12 \int_{\GC} \chi |\uu|^2 \nlocss\chi \dd x +
  \int_{\GC} \left(\frac12 |\nabla\chi|^2 + \widehat{\beta}(\chi){+} \gamma(\chi) \right) \dd  x\,.\end{aligned}
\end{equation}
In \eqref{gen-gra-sys}, $  \partial \calR : \bsV \times H \rightrightarrows \bsV' \times H $ is the subdifferential of $\calR$ in the sense of convex analysis, while
$ \partial^- \calE :  \bsV \times H \rightrightarrows \bsV' \times H $ is the \emph{Fr\'echet subdifferential} (cf., e.g., \cite{Mordu}) of $\calE$. % with respect to the variables $(\uu,\chi)$.
 We will not specify its definition, here, but only mention that, in the case of the specific energy functional driving system  \eqref{PDE}, $ \partial^- \calE$ is given by the sum of the convex analysis subdifferentials of the convex contributions to $\calE$, with the G\^ateau-derivatives of the nonconvex, but smooth contributions. Thus,
%so that
 \eqref{gen-gra-sys} yields  \eqref{PDE}.
%%%%%%%
\par
However,
 for technical reasons, we will not directly exploit
the structure \eqref{gen-gra-sys} in the proof of our existence result for system \eqref{PDE}.  Nonetheless,  \eqref{gen-gra-sys}  underlies  \eqref{enid-balance}. 
%, holding along any solution
%to Problem 2.2.
\end{remark}
%%%%%%
\begin{remark}
\upshape
No uniqueness result seems to be available for Problem \ref{prob:PDE}, due to the doubly nonlinear character of the flow rule for  $\chi$. Nonetheless, arguing in the very same way as in the proof of \cite[Prop.\ 2.3]{BBR1} (cf.\ also the proof of Prop.\ \ref{prop:loc-exis} ahead), it should be possible to prove a \emph{local-in-time} uniqueness result,  in the particular (physical) case of $\beta=\partial I_{[0,1]}$. Namely, that any two quintuples of solutions $(\uu_1,\zzeta_1,\chi_1,\omega_1, \xi_1)$ and $(\uu_2,\zzeta_2,\chi_2,\omega_2, \xi_2)$ such that there exists $T_0\in (0,T]$ with
\[
0<\chi_i(x,t)<1 \qquad \text{for all } (x,t)\in \overlineGC \times [0,T_0] \quad \text{for } i=1,2,
\]
do coincide on $[0,T_0]$. In turn, the above separation property is guaranteed as soon as the initial datum $\chi_0$ fulfills
%$\min_{x\in \overlineGC} \chi_0(x)>0$
 $\chi_0\in [\delta,1)$ on $\GC$, for some $\delta\in(0,1)$ 
(note that \eqref{ini-chi} allows for $\chi_0$ taking the  values $0$ and $1$, instead).
\end{remark}
\noindent
 The \emph{proof} of Theorem \ref{th:main}
will be developed throughout Sections \ref{s:3}--\ref{s:5} by setting up an approximate system and proving the existence (and uniqueness) of local-in-time solutions for it (Sec.\ \ref{s:3}),
passing to the limit  with the approximation parameter and concluding the existence of local solutions for Problem \ref{prob:PDE} (Sec.\ \ref{ss:3.3}), extending the local solution to a global one via a non-standard prolongation argument (Sec.\ \ref{s:5}). 
%%%%
\section{The approximate system}
\label{s:3}
 After introducing the approximate problem in Sec.\ \ref{ss:3.1}, in Sec.\ \ref{ss:3.2}
we prove its local-in-time well-posedness by means of a fixed point argument combined with continuous dependence estimates. 
\subsection{Setup of the approximate problem}
\label{ss:3.1}
In system \eqref{PDE-approx} below, we  will
 replace  the operators $\rho,\,  \beta : \R \rightrightarrows \R$ featuring in the flow rule for $\chi$
 by their Yosida
regularizations
$\rho_\eps $ and $\beta_\eps$, cf.\ e.g.\
\cite{brezis73}.
 We will exploit  that $\rho_\eps,\,\beta_\eps:\R \to \R$  are Lipschitz continuous functions on $\R$, and denote by $\widehat{\rho}_\eps$ and $\widehat{\beta}_\eps$ their primitives fulfilling
 $\widehat{\rho}_\eps (0) = 0$ and  $\widehat{\beta}_\eps (0) =0$. The convex, $\rmC^2$-functions  $\widehat{\rho}_\eps $ and  $\widehat{\beta}_\eps $ are in fact the Yosida approximations of the functions $ \widehat{\rho}$ and $\widehat{\beta}$, defined by
 %\footnote{\RRR qui ho richiamato la def.\ delle approssimate di Yosida.. nel caso possa servire..}
 \begin{equation}
 \label{Yos-reg}
 \widehat{\rho}_\eps (y): = \min_{x\in \R} \left( \frac1{2\eps} |y{-}x|^2 + \widehat{\rho}(x)\right) \qquad
 \widehat{\beta}_\eps(y):= \min_{x\in \R} \left( \frac1{2\eps} |y{-}x|^2 + \widehat{\beta}(x)\right) \,.
 \end{equation}
% For later use, we record the explicit expressions
 %\begin{equation}
 %\label{non-espansive}
 %\rho_\eps(y) = \frac1{\eps}(y)^+, \qquad R_\eps (r): = (\mathrm{Id}+\rho_\eps)^{-1}(r) = -(r)^- + \frac{\eps}{\eps+1}(r)^+ \qquad \text{for every } r \in \R
 %\end{equation}
% and observe that $R_\eps: \R \to \R$ is nonexpansive.
\par
We are now in a position to give the  variational  formulation of the approximate problem.
\begin{problem}
\label{prob:approx}
Starting from initial data $(\uu_0,\chi_0)$ fulfilling \eqref{hyp:initial},
find a triple
%\footnote{\RRR problema di LaTeX: vorrei scrivere $\zeta$ in bold ma non ci riesco...}
 $(\uu,\zzeta,\chi)$ as in \eqref{reg-u}--\eqref{reg-chi}
such that $\uu(0)=\uu_0$, $\chi(0)=\chi_0$,
and fulfilling a.e.\ in (0,T)
\begin{subequations}
\label{PDE-approx}
\begin{align}
&  \label{eqU-approx}
\! \!\!\!\!\!\!
b(\uu_t,\mathbf{v}) + a(\uu,\mathbf{v}) +
 \int_{\GC} \chi\uu \mathbf{v} \dd x +
 \langle \zzeta,\mathbf{v} \rangle_{\bsY}
+  \int_{\GC} \chi\uu \nlocss{\chi} \mathbf{v} \dd x =
 \pairing{}{\bsV}{\mathbf{F}}{\vv} && \text{for all } \vv \in \bsV,
 \\
 &
 \label{zeta-approx}
 \! \!\!\!\!\!
% \! \!\!\!\!\!\!\!\!\!
% \! \!\!\!\!\!\!\!\!\!\!\!
 \zzeta \in \balpha(\uu) && \text{in } \bsY',
 \\
 &
 \label{eqCHI-eta}
 \! \!\!\!\!\!\!
% \! \!\!\!\!\!\!\!\!\!\!\!
 \chi_t + \rho_\eps(\chi_t) -\Delta\chi + \beta_\eps(\chi)+ \gamma'(\chi) = -\frac12 |\uu|^2 - \frac12 |\uu|^2 \nlocss{\chi}  -\frac12 \nlocss{\chi|\uu|^2} && \text{ a.e.\ on } \GC.
 \end{align}
\end{subequations}
\end{problem}
\begin{remark}
\label{rmk:features-reg-prob}
\upshape
On the one hand,
with these regularizations, we will able to render all a priori estimates rigorously within the frame of the approximate system. In particular, we refer to a \emph{regularity} estimate that  consists in
testing the  flow rule \eqref{eqCHI-eta} by $\partial_t (-\Delta\chi + \beta_\eps(\chi))$
 (cf.\ the proof of Proposition \ref{prop:loc-exis}). This estimate 
  allows us to bound both $-\Delta\chi$ and (a selection) $\xi \in \beta(\chi)$ in $L^\infty(0,T;H)$, thus giving rise to
\eqref{further-reg} (since the  estimates for $\chi_t$ and $\omega \in \rho(\chi_t)$ then follow by comparison in \eqref{eqCHI}).  On the other hand,
 since
$\mathrm{dom}(\beta_\eps)=\R$,
 in the approximate system \eqref{PDE-approx}
 the constraint $\chi \geq 0$ will be  no longer enforced.
 Because of this, the first a priori estimate performed  on the approximate system
  (cf.\ again
 the proof of Proposition \ref{prop:loc-exis})  will not have a \emph{global-in-time} character, and the procedure for extending a local solution to a global one, developed in Sec.\ \ref{s:5},  will be more complex than the standard one.
 \end{remark}
\subsection{Local-in-time existence for the approximate problem}
\label{ss:3.2}
This section is devoted to the proof of the following result.
\begin{proposition}
\label{prop:local}
Assume
\eqref{hyp:GC}, \eqref{ass-K}, \eqref{hyp-alpha}--\eqref{hyp-gamma},
and \eqref{hyp:forces}. Let $(\uu_0,\chi_0)$ be a pair of initial data fulfilling
\eqref{hyp:initial}.
Then,  there exists a final time $\widehat{T}>0$ such that, for every $\eps>0$,  Problem \ref{prob:approx} admits a unique   solution $(u,\zzeta,\chi)$ on $(0,\widehat{T})$,
 with the enhanced regularity $\chi \in L^\infty (0,\widehat{T};W) \cap H^1(0,\widehat{T};V) \cap W^{1,\infty}(0,\widehat{T};H)$. 
\end{proposition}
We will prove
Proposition \ref{prop:local} by
constructing an operator, defined  between suitable function spaces,
whose fixed points yield solutions to system \eqref{PDE-approx}, and by showing that it does admit fixed points thanks to the Schauder theorem.
This procedure will yield a local-in-time solution to Problem \ref{prob:approx}  defined on the interval $(0,\widehat{T})$, with $\widehat{T}$  in fact \emph{independent} of $\varepsilon$, cf.\ Remark \ref{rmk:indep-eps} ahead. In view of this, in  Sec.\ \ref{ss:3.3} we will  pass to the limit in system \eqref{PDE-approx} as $\eps \downarrow 0$ and obtain the existence of local-in-time solutions to the original Problem \ref{prob:PDE}.
\par
Our construction  of the Schauder operator will be based on  two  results,
Lemma \ref{l:solvability-u}  and Lemma \ref{l:solvability-chi} ahead,
tackling the separate solvability,   on a given interval $(0,\sfT)$,
of the momentum balance \eqref{eqU-approx}, with  $\chi$ replaced by a \emph{given} function $\overline{\chi}\in  L^4(0,\sfT;H)$, and of the flow rule \eqref{eqCHI-eta} for the adhesive parameter, with (the trace of) $\uu$ replaced by a given  $\overline{\uu} \in H^1(0,\sfT;L^4(\GC)) $. 
Preliminarily, we fix the properties of the nonlocal operator $\nlname$ from \eqref{nl-op-intro}  in the following result, whose proof is left to the reader.
%\footnote{\RRR Mi sembra che non valga la pena di dettagliarla..}
\begin{lemma}\label{lemmaK}
Assume  \eqref{hyp-k}. Then, $\nlname$  is well defined, linear and continuous from $L^1(\GC)$ to $L^\infty(\GC)$,
with
\begin{equation}
\label{stima-puntuale}
|\nlocs w x | \leq \| k(x,\cdot)\|_{L^\infty(\GC)} \|w\|_{L^1(\GC)} \  \text{for a.a. } x \in \GC,  \text{ so that }  \| \nlocss w\|_{L^\infty(\GC)} \leq \| k \|_{L^\infty(\GC\times \GC)}  \|w\|_{L^1(\GC)}\,.
\end{equation}
 Furthermore,  for every $1\leq p<\infty$  the operator $\nlname$ is
 continuous from $L^1(\GC)$, equipped with the weak topology, to  $L^p(\GC)$  with the strong topology 
  %\emph{weakly \GGG strongly \EEE continuous} from to $L^p(\GC)$ for every.
   (i.e.\  if $w_n\debole w$ in $L^1(\GC)$ then $\nlocss{w_n}\rightarrow \nlocss{w}$ in $L^p(\GC)$).  Finally, there holds
\begin{equation}
\label{nl-symm}
\int_{\GC} \nlocs {w_1}x \, w_2(x) \dd  x = \int_{\GC} \nlocs {w_2}x \, w_1(x) \dd x \qquad \text{for all } w_1,\, w_2 \in L^1(\GC)\,.
\end{equation}
\end{lemma}
%%%%
\par
We start by tackling
the momentum balance \eqref{eqU-approx}.   In what follows we will denote by $
Q_i$ and $\widetilde{Q}_i$, $i=1,\ldots$, computable,  non-negative and continuous functions, monotone increasing w.r.t.\ each of their variables, that will enter in the a priori estimates holding for the solutions to the momentum balance/adhesive flow rule.
\begin{lemma}
\label{l:solvability-u}
Assume
\eqref{hyp:GC},
 \eqref{ass-K}, \eqref{hyp-alpha},   \eqref{hyp-k},
and \eqref{hyp:forces}. Let $\uu_0$ fulfill \eqref{ini-u}.
Then, for every $\overline{\chi}\in L^4(0,\sfT;H)$ there exists a unique  pair $(\uu, \zzeta) \in H^1(0,\sfT;\bsV)\times L^2(0,\sfT;\bsY')$  fulfilling $\uu(0) = \uu_0$ and
\begin{equation}
\label{eq-u-fx-point}
b(\uu_t,\mathbf{v}) + a(\uu,\mathbf{v}) +
 \int_{\GC} \overline{\chi}\uu \mathbf{v} \dd x +
 \langle \zzeta,\mathbf{v} \rangle_{\bsY}
+  \int_{\GC} \overline{\chi}\uu \nlocss{\overline{\chi}} \mathbf{v} \dd x =
 \pairing{}{\bsV}{\mathbf{F}}{\vv} \quad \text{for all } \vv \in \bsV
 \end{equation}
 for a.a. $t \in (0,\sfT),$ with
  \begin{equation}
 \label{est-u-fix-pt}
 \| \uu\|_{H^1(0,\sfT;\bsV)} \leq Q_1 (\|\uu_0\|_{\bsV},  \widehat\balpha(\uu_0),  \|\mathbf{F}\|_{L^2(0,\sfT;\bsV')},\|\overline{\chi}\|_{L^4(0,\sfT;H)})\,.
 \end{equation}
Furthermore, there exists a positive function $\widetilde{Q}_1$ such that for every $\overline{\chi}_1\,, \overline{\chi}_2 \in  L^4(0,\sfT;H)$, with $\uu_1,\,\uu_2$ the associated solutions of \eqref{eq-u-fx-point} starting from $\uu_0$, there holds
\begin{equation}
\label{est-u-cont-dep}
 \|\uu_1-\uu_2\|_{L^\infty(0,\sfT;\bsV)}  \leq
 \widetilde{Q}_1\left(
 \max_{i=1,2} \|\uu_i\|_{L^\infty(0,\sfT;L^4(\GC))},  \max_{i=1,2} \|\chi_i\|_{L^4(0,\sfT;H)}\right)  \|\overline{\chi}_1 - \overline{\chi}_2 \|_{L^4(0,\sfT;H)}
  \,.
 \end{equation}
\end{lemma}
We denote by
\begin{equation}
\label{sol-op-chi}
\mathcal{T}_1:  L^4(0,\sfT;H) \to  H^1(0,\sfT;\bsV) \quad \text{the solution operator associated with  \eqref{eq-u-fx-point}}.
\end{equation}
\begin{proof}
A standard modification
%\footnote{\RRR infatti, in \cite{BBR1} partivamo da una $\overline\chi \in L^2(0,\sfT;L^p(\GC))$ con $p\in (2,4)$. Questo veniva usato solo nella dimostrazione della convergenza *forte* di $\chi_k \uu_k$ (le soluzioni approssimate), che usavamo per passare al limite nell'equazione degli spostamenti. Ma per fare ci\`o dovrebbe bastare la convergenza debole di questo prodotto..}
 of the proof of \cite[Prop.\ 4.2]{BBR1},  which is  in turn based on a  time-discretization procedure, yields the existence statement. Estimate \eqref{est-u-fix-pt} follows from testing \eqref{eq-u-fx-point} by $\uu_t$ and integrating on a time interval $(0,t)$,
which gives
\begin{equation}
\label{mom-bal-enid}
\begin{aligned}
&
\int_0^t b(\uu_t,\uu_t) \dd r + \frac12 a(\uu(t),\uu(t)) + \bvarphi(\uu(t))  \\ &  =  \frac12 a(\uu_0,\uu_0) + \bvarphi(\uu_0) +  \int_0^t  \pairing{}{\bsV}{\mathbf{F}}{\uu_t}  \dd r -  \int_0^t \int_{\GC} \overline{\chi} \uu \uu_t \dd x \dd r -    \int_0^t  \int_{\GC} \overline{\chi} \uu \nlocss{\overline\chi} \uu_t \dd x \dd r
\end{aligned}
\end{equation}
thanks to the chain-rule formula
\begin{equation}
\label{ch-rule-alpha}
\int_0^t \langle \zzeta, \uu_t  \rangle_{\bsY} \dd r  = \widehat\balpha(\uu(t)) - \widehat\balpha(\uu_0)\,.
\end{equation}
In view of \eqref{korn} and of the positivity of $\bvarphi$,
we have that
\[
\text{left-hand side of \eqref{mom-bal-enid}} \geq C_b \int_0^t \| \uu_t\|_{\bsV}^2 \dd r + \frac{C_a}{2} \| \uu(t)\|_{\bsV}^2\,.
\]
Moreover, applying the Young inequality, we get
\begin{equation}
\label{added-4-later}
\begin{aligned}
&
\left| \int_0^t  \pairing{}{\bsV}{\mathbf{F}}{\uu_t}\dd r  \right|
\leq \frac{C_b}4\int_0^t \|\uu_t\|_{\bsV}^2 \dd r + C \|\mathbf{F}\|_{L^2(0,\sfT;\bsV')}^2.
\end{aligned}
\end{equation}
 We then perform  the following estimates for the last two terms on the right-hand side of
\eqref{mom-bal-enid}
\[
\begin{aligned}
&
\left| \int_0^t \int_{\GC} \overline{\chi}\uu \uu_t \dd x  \dd r  \right|
\leq \frac{C_b}4\int_0^t \|\uu_t\|_{\bsV}^2 \dd r + C \int_0^t \|\overline\chi\|_{H}^2 \|\uu\|_{\bsV}^2 \dd r,
\\
&
\left| \int_0^t  \int_{\GC} \overline{\chi}\uu \nlocss{\overline{\chi}} \mathbf{u}_t \dd x  \dd r \right| \leq \frac{C_b}4\int_0^t \|\uu_t\|_{\bsV}^2 \dd r + C \int_0^t \|\overline\chi\|_{H}^4 \|\uu\|_{\bsV}^2 \dd r,
\end{aligned}
\]
where we have used \eqref{embedding}, %\footnote{$\bsV\subset L^4(\GC)$},
as well as estimate \eqref{stima-puntuale}, yielding
\begin{equation}
\label{nonloc-later}
 \| \nlocss{\chi}\|_{L^\infty(\GC)} \leq C \|\chi\|_{L^1(\GC)} \leq C \|\chi\|_{H}.
\end{equation}
 %$\| \nlocss{\overline{\chi}} \|_{L^\infty(\GC)}$.
All in all, we obtain
\[
\frac{C_b}4\int_0^t \|\uu_t\|_{\bsV}^2 \dd r  + \frac{C_a}2 \|\uu(t)\|_{\bsV}^2
\leq   C\left(\|\uu_0\|_{\bsV}^2 + \widehat\balpha(\uu_0)+ \|\mathbf{F}\|_{L^2(0,\sfT;\bsV')}^2 + \int_0^t \|\overline\chi\|_{H}^2 \|\uu\|_{\bsV}^2 \dd r +\int_0^t \|\overline\chi\|_{H}^4 \|\uu\|_{\bsV}^2 \dd r \right),
\]
whence \eqref{est-u-fix-pt} by the Gronwall Lemma.
\par
In order to show the continuous dependence estimate \eqref{est-u-cont-dep}, we subtract \eqref{eq-u-fx-point} for given $\overline{\chi}_2$ from \eqref{eq-u-fx-point} for  $\overline{\chi}_1$ and test the resulting equation by $\uu_1 -\uu_2$. With calculations similar to those above, we  obtain,   by virtue of \eqref{stima-puntuale} and  the Young inequality,
%as well as Lemma \ref{lemmaK},
 that
\begin{equation}
\label{quoted4cont-dep}
\begin{aligned}
&
\frac{C_b}2 \| (\uu_1 -\uu_2)(t)\|_{\bsV}^2 +\frac{C_a}2 \int_0^t \| \uu_1-\uu_2\|_{\bsV}^2 \dd r
\\&  \leq
 \frac{C_a}4 \int_0^t \| \uu_1-\uu_2\|_{\bsV}^2 \dd r   + C  \int_0^t \left(1  +  \|\overline{\chi}_1 \|_{H}+  \|\overline{\chi}_1 \|_{H}^2\right)  \|  \uu_1-\uu_2\|_{L^4(\GC)}^2 \dd r
 \\ & \qquad \qquad
 + C
 \int_0^t \| \uu_2\|_{L^4(\GC)}^2 (1+  \|\overline{\chi}_1 \|_{H}^2 + \|\overline{\chi}_2 \|_{H}^2)    \|\overline{\chi}_1-\overline{\chi}_2 \|_H^2 \dd r\,,
 \end{aligned}
\end{equation}
whence \eqref{est-u-cont-dep}.
%\footnote{\RRR per favore, ricontrollatela, a me sembra che venga cosi'..} by the Gronwall Lemma and Sobolev's embeddings.
  Finally, a comparison in \eqref{eq-u-fx-point}  yields the uniqueness of $\zzeta$, too.
\end{proof}
\par
We now address the flow rule
\eqref{eqCHI-eta}
for the adhesion parameter.
\begin{lemma}
\label{l:solvability-chi}
Assume
\eqref{hyp:GC},
  \eqref{hyp-rho}--\eqref{hyp-gamma}. Let $\chi_0$ fulfill \eqref{ini-chi}. Then, for every  $\overline{u} \in H^1(0,\sfT;L^4(\GC;\R^3))$ 
  there exists a unique   $\chi\in L^\infty(0,\sfT;W) \cap H^1(0,\sfT;V) \cap  W^{1,\infty}(0,\sfT;H)$  fulfilling $\chi(0)=\chi_0$ and
\begin{equation}
\label{flow-rule-u-fixed}
 \chi_t + \rho_\eps(\chi_t) -\Delta\chi+ \beta_\eps(\chi)
  + \gamma'(\chi) = -\frac12 |\overline\uu|^2 - \frac12 |\overline\uu|^2 \nlocss{\chi}  -\frac12 \nlocss{\chi|\overline\uu|^2}  \qquad \text{a.e.\ on } \   \GC\times (0,\sfT),
 \end{equation}
 with
  \begin{equation}
 \label{est-chi-fix-pt}
 \|\chi\|_{L^2(0,\sfT;W) \cap L^\infty(0,\sfT;V) \cap H^1(0,\sfT;H)} \leq Q_2 (\|\chi_0\|_{V}, \|\widehat{\beta}(\chi_0)\|_{L^1(\GC)},
 \|\overline{\uu}\|_{L^4(0,\sfT;L^4(\GC))})\,.
 \end{equation}
 Furthermore, there exists a function $\widetilde{Q}_2$ such that for every $\overline{\uu}_1,\, \overline{\uu}_2 \in  L^4(0,\sfT;L^4(\GC;\R^3))$, with $\chi_1,\, \chi_2$ the associated solutions to \eqref{flow-rule-u-fixed} emanating from the same initial datum $\chi_0$, there holds
 \begin{equation}
 \label{est-chi-cont-depen}
   \|\chi_1-\chi_2\|_{L^\infty(0,\sfT;V) \cap H^1(0,\sfT;H)}  \leq
 \widetilde{Q}_2\left( \eps^{-1},
 \max_{i=1,2} \|\overline{\uu}_i\|_{L^4(0,\sfT;L^4(\GC))},  \max_{i=1,2} \|\chi_i\|_{L^\infty(0,\sfT;H)}\right)   \|\overline{\uu}_1 - \overline{\uu}_2 \|_{L^2(0,\sfT;L^4(\GC))}\,. 
  \end{equation}
\end{lemma}
We denote by
\begin{equation}
\label{sol-op-u}
\begin{gathered}
 \mathcal{T}_2:  H^1(0,\sfT;L^4(\GC;\R^3)) \to  %L^2(0,\sfT;W) \cap L^\infty(0,\sfT;V) \cap  H^1(0,\sfT;H)
 L^\infty(0,\sfT;W) \cap H^1(0,\sfT;V) \cap  W^{1,\infty}(0,\sfT;H)  \\ \text{the solution operator associated with  \eqref{flow-rule-u-fixed}}.
\end{gathered}
\end{equation}
\begin{proof}
Observe that \eqref{flow-rule-u-fixed} can be recast as the abstract doubly nonlinear equation
\begin{equation}
\label{DNE}
\partial \mathcal{R}_{\mathsf{\chi}}^\eps(\chi_t(t)) + \partial^- \mathcal{E}_2(\overline{\uu}(t),\chi(t)) \ni 0 \qquad \text{in } H \quad \foraa\, t \in (0,\sfT),
\end{equation}
with the dissipation potential $ \mathcal{R}_{\mathsf{\chi}}^\eps(\dot \chi): = \tfrac12 \| \dot \chi\|_H^2 + \int_{\GC} \widehat{\rho}_\eps(\dot \chi) \dd x $
(cf.\ \eqref{underlying-diss}), and the driving energy $\mathcal{E}_2$ from  \eqref{underlying-en}. In fact, here both the convex analysis subdifferential $\partial \mathcal{R}_{\mathsf{\chi}}^\eps$ and the Fr\'echet subdifferential $\partial^- \mathcal{E}_2$ reduce to singletons.
Under the %additional, technical
 condition  that $\overline\uu \in H^1(0,\sfT;L^4(\GC;\R^3))$, we may apply  \cite[Theorem 2.2]{MRS} to conclude the existence of a solution
 $\chi \in L^2(0,\sfT;W) \cap L^\infty(0,\sfT;V) \cap  H^1(0,\sfT;H) $
 for (the Cauchy problem for)  \eqref{DNE},
 via a time discretization procedure. The enhanced regularity $ L^\infty(0,\sfT;W) \cap H^1(0,\sfT;V) \cap  W^{1,\infty}(0,\sfT;H) $ may be inferred by
performing, on the time discrete level,  the \emph{regularity estimate} that consists in testing the flow rule by $\partial_t (-\Delta\chi +\beta_\eps(\chi))$ (cf.\ the proof of
Prop.\ \ref{prop:loc-exis}). 
%The existence statement under the weaker condition $\overline\uu \in L^4(0,\sfT;L^4(\GC;\R^3))$ follows from a standard approximation procedure.
\par
 Estimate \eqref{est-chi-fix-pt}
follows from testing \eqref{flow-rule-u-fixed}
 by $\chi_t$ and integrating in time. Taking into account that
 \[
 \int_0^t \int_{\GC} \rho_\eps(\chi_t) \chi_t \dd x \dd r \geq 0,
 \]
 since $\rho_\eps(0)=0$ and $\rho_\eps$ is increasing,
  and using that
 $\int_0^t \int_{\GC} \beta_\eps(\chi) \chi_t  = \int_{\GC}\widehat{\beta}_\eps(\chi(t)) \dd x - \int_{\GC}\widehat{\beta}_\eps(\chi_0) \dd x $ by the chain rule, we get
 \[
 \int_0^t \|\chi_t\|_{H}^2 \dd r  + \frac12 \| \nabla \chi(t)\|_{H}^2 + \int_{\GC}\widehat{\beta}_\eps(\chi(t)) \dd x
 \leq
 \frac12 \|\nabla \chi_0\|_H^2 +
 \int_{\GC}
  \widehat{\beta}(\chi_0) \dd x +I_1+I_2+I_3+I_4,
 \]
 where we have also used that $   \int_{\GC}
 \widehat{\beta}_\eps(\chi_0) \dd x
 \leq  \int_{\GC}
 \widehat{\beta}(\chi_0) \dd x $.
 With Young's inequality, we estimate
 \[
 \begin{aligned}
 &
 \begin{aligned}
 |I_1|\leq \int_0^t \int_{\GC} |\gamma'(\chi)||\chi_t| \dd x  \dd r &  \stackrel{(1)}{\leq}
  \int_0^t \int_{\GC} (|\gamma'(\chi_0)|+C |\chi-\chi_0|)|\chi_t| \dd x  \dd r
\\ &   \leq C \|\chi_0\|_H^2 +  \frac18   \int_0^t \|\chi_t\|_{H}^2 \dd r + C  \int_0^t \|\chi\|_{H}^2 \dd r  + C,
\end{aligned}
  \\
  &
  |I_2| \leq  \int_0^t \int_{\GC} |\overline\uu|^2|\chi_t| \dd x  \dd r \leq
   \frac18   \int_0^t \|\chi_t\|_{H}^2 \dd r +2 \int_0^t     \| \overline\uu\|_{L^4(\GC)}^4   \dd r,
   \\
   &
   \begin{aligned}
   |I_3| \leq  \int_0^t \int_{\GC} |\overline\uu|^2 | \nlocss{\chi} ||\chi_t| \dd x  \dd r & \leq  \frac18   \int_0^t \|\chi_t\|_{H}^2 \dd r  + C \int_0^t \| \overline\uu\|_{L^4(\GC)}^4 \| \nlocss{\chi}\|_{L^\infty(\GC)}^2 \dd r
   \\ &
   \stackrel{(2)}{\leq}
  \frac18   \int_0^t \|\chi_t\|_{H}^2 \dd r  + C \int_0^t \| \overline\uu\|_{L^4(\GC)}^4 \| \chi\|_{H}^2 \dd r,
  \end{aligned}
  \\
  &
   |I_4| \leq  \int_0^t \int_{\GC}|  \nlocss{\chi|\overline\uu|^2} ||\chi_t| \dd x  \dd r    \stackrel{(3)}{\leq}  \frac18   \int_0^t \|\chi_t\|_{H}^2 \dd r  + C \int_0^t \| \overline\uu\|_{L^4(\GC)}^4 \|\chi\|_{H}^2 \dd r,
 \end{aligned}
 \]
 with (1) due to the Lipschitz continuity of $\gamma'$, (2) to estimate \eqref{nonloc-later},
  and (3) again following from  \eqref{stima-puntuale}, via $\| \nlocss{\chi |\overline\uu|^2}\|_{L^\infty(\GC)} \leq C
 \|\chi |\overline\uu|^2\|_{L^1(\GC)} \leq C \|\overline\uu\|_{L^4(\GC)}^2 \|\chi\|_{H}$.
%, and (2) from
% $  \|(\chi_1)^+ - (\chi_2)^+\|_{L^1(\GC)} \leq
 %\|(\chi_1 - \chi_2)^+\|_{L^1(\GC)}  \leq \|\chi_1 - \chi_2\|_{L^1(\GC)} $.
  Moreover, a comparison in \eqref{flow-rule-u-fixed} also provides a bound for $\|\Delta\chi\|_{L^2(0,\sfT;H)}$, whence the estimate for $\chi$ in $L^2(0,\sfT;W)$ by standard elliptic regularity results, relying on the assumption of $\mathrm{C}^{1,1}$-boundary for $\GC$. 
 \par
 Finally, in order to prove \eqref{est-chi-cont-depen},
 let us  preliminarily introduce,
 for fixed $\overline\uu \in  L^4(0,\sfT;L^4(\GC;\R^3))$,
the \emph{Lipschitz continuous} mapping
\[
\mathcal{F}(\overline\uu; \cdot): H \to H \text{ defined by }
\mathcal{F}(\overline\uu;\chi): =
 - \beta_\eps(\chi) - \gamma'(\chi) -\frac12 |\overline\uu|^2 - \frac12 |\overline\uu|^2 \nlocss{\chi}  -\frac12 \nlocss{\chi |\overline\uu|^2}.
 \]
Then,
  we  subtract  equation \eqref{flow-rule-u-fixed}, written for a given  $\overline{\uu}_2$ in $L^4(0,\sfT; L^4(\GC;\R^3))$,
from \eqref{flow-rule-u-fixed}  for $\overline{\uu}_1$,  test the resulting relation by $\partial_t(\chi_1-\chi_2)$.
We use that
\[
\int_{\GC} \left(\rho_\eps(\partial_t\chi_1){-}\rho_\eps(\partial_t\chi_2) \right) (\partial_t\chi_1{-}\partial_t\chi_2) \dd x  \geq 0
\]
 a.e.\ in $(0,\sfT)$ 
by the monotonicity of $\rho_\eps$,
 and, on the right-hand side, we  estimate
\[
\left| \int_{\GC} \left( \mathcal{F}(\overline{\uu}_1;\chi_1) {-}\mathcal{F}(\overline{\uu}_2;\chi_2)  \right)  \partial_t(\chi_1-\chi_2) \dd x \right| \leq \frac12  \|\partial_t(\chi_1-\chi_2)\|_{H}^2
+ \frac12 \|\mathcal{F}(\overline{\uu}_1;\chi_1) -\mathcal{F}(\overline{\uu}_2;\chi_2) \|_H^2\,.
\] 
All in all, we obtain
\begin{equation}
\label{quoted-later}
\begin{aligned}
 & \frac12\|\partial_t(\chi_1-\chi_2)\|_{H}^2
+\frac12 \frac{\dd}{\dd t} \|\nabla(\chi_1-\chi_2)\|_H^2
\stackrel{(1)}{\leq}
\frac12 \|\mathcal{F}(\overline{\uu}_1;\chi_1) -\mathcal{F}(\overline{\uu}_2;\chi_2) \|_H^2
\\ &
 \leq \left(C+\frac1\eps \right) \|\chi_1-\chi_2\|_H^2 +   C (\|\overline{\uu}_1\|_{L^4(\GC)}^2 + \|\overline{\uu}_2\|_{L^4(\GC)}^2 ) \|\overline{\uu}_1 -\overline{\uu}_2\|_{L^4(\GC)}^2
 \\
 & \qquad + C  \|\chi_1-\chi_2\|_H^2\|\overline{\uu}_2\|_{L^4(\GC)}^4
 +
C\|\chi_1\|_H^2  (\|\overline{\uu}_1\|_{L^4(\GC)}^2 + \|\overline{\uu}_2\|_{L^4(\GC)}^2 )  \|\overline{\uu}_1 -\overline{\uu}_2\|_{L^4(\GC)}^2.
\end{aligned}
 \end{equation}
 a.e.\ in $(0,\sfT)$. 
 For (1), we have used the Lipschitz continuity of $\gamma'$ and of $\beta_\eps$ (with Lipschitz constant $\tfrac1\eps$), and we have for example estimated
 \[
\begin{aligned}
&
\|  |\overline{\uu}_1|^2 \nlocss{\chi_1}   -  |\overline{\uu}_2|^2 \nlocss{\chi_2} \|_{H}^2
\\
&
  \leq
 2\| \overline{\uu}_1+ \overline{\uu}_2\|_{L^4(\GC)}^2
 \| \overline{\uu}_1- \overline{\uu}_2\|_{L^4(\GC)}^2 \|\nlocss{\chi_1}\|_{L^\infty(\GC)}^2 +  2
 \|\overline{\uu}_2\|_{L^4(\GC)}^4  \|\nlocss{\chi_1}- \nlocss{\chi_2}\|_{L^\infty(\GC)}^2
  \\
  &\stackrel{(2)}{\leq}  C
   \| \overline{\uu}_1+ \overline{\uu}_2\|_{L^4(\GC)}^2
 \| \overline{\uu}_1- \overline{\uu}_2\|_{L^4(\GC)}^2 \|\chi_1\|_{H}^2
 +  C \| \overline{\uu}_2\|_{L^4(\GC)}^4  \|\chi_1- \chi_2\|_{H}^2
\end{aligned}
\]
where (2) follows from \eqref{stima-puntuale}.
Therefore, from \eqref{quoted-later}  we deduce, via  the Gronwall Lemma, that
 \[
 \begin{aligned}
\|\partial_t (\chi_1-\chi_2)(t)\|_H^2 \leq C & \mathrm{exp}\left(C\left(1+\frac1\eps\right)t + C \int_0^{t}\|\overline{\uu}_2\|_{L^4(\GC)}^4 \right)
\\
& \times
 (\|\chi_1\|_H^2+1) (\|\overline{\uu}_1\|_{L^4(\GC)}^2 + \|\overline{\uu}_2\|_{L^4(\GC)}^2 ) \|\overline{\uu}_1 -\overline{\uu}_2\|_{L^4(\GC)}^2
 \end{aligned}
 \]
for almost all $t\in (0,\sfT)$. Upon integrating in time we then conclude  estimate \eqref{est-chi-cont-depen} for  $\| \chi_1-\chi_2\|_{H^1(0,\sfT;H)}$.

  We then also recover the bound for  $\| \chi_1-\chi_2\|_{L^\infty(0,\sfT;V)}$.
 \end{proof}
\paragraph{\bf Solution operator for Problem \ref{prob:approx} and application of the Schauder theorem.}
For  fixed $M>0$ %and $p\in (2,4)$,\footnote{\RRR in BBR1 avevamo scelto $L^2(0,T;
%L^p(\GC))$ come spazio per il punto fisso.. ora non vedo piu' l'utilita' di lavorare in $L^p(\GC)$, credo che basterebbe stare in $L^4(0,T;H)$... che dite?? \ELE in effetti non lo uso...pensiamoci ancora \EEE}
we consider the closed ball
\begin{equation}
\label{ball-fixed-point}
\mathcal{S}_{M}(\sfT): = \{ \chi\in  L^4(0,\sfT;H)\, : \| \chi\|_{L^4(0,\sfT; H)}  \leq M \}\,.
\end{equation}
In view of \eqref{sol-op-chi}  and \eqref{sol-op-u}, the operator
\begin{equation}
\label{sol-op-comp}
\mathcal{T}: = \mathcal{T}_2 \circ \mathcal{T}_1:   L^4(0,\sfT;H) \to  L^\infty(0,\sfT;W) \cap H^1(0,\sfT;V) \cap  W^{1,\infty}(0,\sfT;H) \qquad \text{ is well defined}.
\end{equation}
Our next result shows that, at least for small times, the map $\mathcal{T}$  complies with the conditions of the Schauder theorem.
\begin{lemma}
\label{l:schauder}
Assume \eqref{hyp:GC}, \eqref{ass-K}, \eqref{hyp-alpha}--\eqref{hyp-gamma}, and \eqref{hyp:forces}. Let the initial data  $(\uu_0,\chi_0)$  fulfill
\eqref{hyp:initial}.
Then, there exists $\widehat{T}>0$ such that
the operator $\mathcal{T}$
\begin{enumerate}
\item
maps $
 \mathcal{S}_{M}(\widehat{T}) $ into itself;
 \item
 is continuous with respect to the strong topology of  $L^4(0,\widehat{T};H) $; 
 \item maps $\mathcal{S}_{M}(\widehat{T}) $ into a compact subset of $L^4(0,\widehat{T};H) $. 
 \end{enumerate}
   % \mathcal{S}_{M}(\widehat{T}) $
\end{lemma}
\begin{proof}
\textbf{Ad (1):} Combining estimates \eqref{est-u-fix-pt} and \eqref{est-chi-fix-pt} we find for every $\sfT\in (0,T]$
\begin{equation}
\label{est-Tchi}
\| \mathcal{T}(\chi)\|_{H^1(0,\sfT;H) \cap L^\infty(0,\sfT;V) } \leq Q_3 (\|\uu_0\|_{\bsV},
 \widehat{\balpha}(\uu_0), 
 \|\chi_0\|_V, \|\widehat{\beta}(\chi_0)\|_{L^1(\GC)}, \|\mathbf{F}\|_{L^2(0,T;\bsV')}, \|\chi\|_{L^4(0,\sfT;H)})\,.
\end{equation}
Therefore,
\[
 \| \mathcal{T}(\chi)\|_{L^4(0,\sfT;H)}  \leq C \sfT^{1/4}  \| \mathcal{T}(\chi)\|_{L^\infty(0,\sfT;V)} \leq  \overline{C}   \sfT^{1/4}  Q_3 (\|\uu_0\|_{\bsV}, \widehat{\balpha}(\uu_0),  \|\chi_0\|_V, \|\widehat{\beta}(\chi_0)\|_{L^1(\GC)}, \|\mathbf{F}\|_{L^2(0,T;\bsV')}, M)\,,
\]
so that property (1) follows by  choosing
\begin{equation}
\label{est-ball-schauder}
0<\widehat{T} \leq  \left( \frac{ M}{ \overline{C} Q_3 (\|\uu_0\|_{\bsV},
 \widehat{\balpha}(\uu_0), 
 \|\chi_0\|_V, \|\widehat{\beta}(\chi_0)\|_{L^1(\GC)}, \|\mathbf{F}\|_{L^2(0,T;\bsV')}, M)} \right)^4 \,. 
\end{equation}
%%%
\textbf{Ad (2):} Let $(\chi_n)_n,\, \chi \subset  \mathcal{S}_{M}(\widehat{T})  $ fulfill $\chi_n\to \chi$ in   $L^4(0,\widehat{T};H)$.  Combining \eqref{est-u-fix-pt} with  \eqref{est-u-cont-dep} we find that $\mathcal{T}_1(\chi_n) \to \mathcal{T}_1(\chi)$ in $L^\infty(0,\widehat{T};\bsV)$.
We then use \eqref{est-chi-fix-pt} and \eqref{est-chi-cont-depen} to conclude that
$\mathcal{T}(\chi_n) = \mathcal{T}_2(\mathcal{T}_1(\chi_n)) \to \mathcal{T}(\chi) = \mathcal{T}_2(\mathcal{T}_1(\chi)) $ in $L^\infty(0,\widehat{T};V) \cap H^1(0,\widehat{T};H)$, and thus in
 $L^4(0,\widehat{T};H)$. 
\par
\noindent
\textbf{Ad (3):}  The compactness property easily follows from combining estimates  \eqref{est-u-fix-pt} and  \eqref{est-chi-fix-pt}.
\end{proof}
\par
Eventually,
 we are in a position to conclude
 the \underline{proof of Proposition \ref{prop:local}}. The existence statement  follows from Lemma \ref{l:schauder} via the Schauder theorem. Uniqueness of solutions to Problem
 \ref{prob:approx} (the same argument
 would also yield continuous dependence on the initial and problem data) ensues by adding up estimates \eqref{quoted4cont-dep} and \eqref{quoted-later} and applying the Gronwall Lemma.  \QED
 %%%%%%%
\begin{remark}
\label{rmk:indep-eps}
\upshape
Estimate \eqref{est-ball-schauder}  shows that
the local-existence time $\widehat{T}$ does not depend on
$\eps$.
\par
A closer examination of the proof of
Lemma \ref{l:schauder} (based on Lemmas \ref{l:solvability-u} and \ref{l:solvability-chi})
 in fact reveals  that $\widehat{T}$ does not depend on
the specific initial conditions $\uu_0$, $\chi_0$ either,
but only on quantities related to them. In other words,
for any ``ball"
of initial data
$$
\mathcal{I}(r):= \left\{ ({\uu}_0, {\chi_0}) \in \bsV \times  W  \, :
\ \| {\uu}_0 \|_{\bsV} +  \widehat{\balpha}(\uu_0) +
  \|{\chi}_0\|_{V} +  
 \|\widehat{\beta}(\chi_0)\|_{L^1(\GC)}  \leq r
\right
\},
$$
there exists a final time $\widehat{T}_r >0$, only depending on $r$
and on known constants, such that, for any  $({\uu}_0,
{\chi_0}) \in \mathcal{I}(r)$ and for any $\eps >0$, the approximate
Problem \ref{prob:approx},
 supplemented with the initial data
$({\uu}_0, {\chi_0})$
  admits a solution
on the interval $(0, \widehat{T}_r )$.
%\end{description}
\end{remark}

\section{Local  existence for Problem \ref{prob:PDE}}
\label{ss:3.3}
%The following result collects all of our a priori estimates on
Throughout this section we will highlight the dependence of the local solution
for Problem  \ref{prob:approx} (found in  Proposition  \ref{prop:local})
on the approximation parameter $\eps$, by denoting it with $(\uu_{\epsi},\zzeta_{\epsi}, \chi_{\epsi})$. In this section, we  perform an asymptotic analysis as
$\epsi\downarrow 0$
 of the sequence
$(\uu_{\epsi},\zzeta_{\epsi}, \chi_{\epsi})_\eps$  on the existence interval
$(0,\widehat{T})$, which \emph{does not} depend on $\eps$, cf.\ Remark \ref{rmk:indep-eps}.
\par
 We thus obtain the following existence result of \emph{local-in-time} solutions for
 Problem \ref{prob:PDE}.
\begin{proposition}
\label{prop:loc-exis}
Assume \eqref{hyp:GC}, \eqref{ass-K},  \eqref{hyp-alpha}--\eqref{hyp-gamma}, 
and \eqref{hyp:forces}. Let $(\uu_0,\chi_0)$   fulfill
\eqref{hyp:initial}.
Then, for every  vanishing   sequence $\eps_k\downarrow 0   $ as $k\to\infty$
 there exists
a %vanishing sequence $\{\epsi_k\}_k$ and a
quintuple $(\uuh,\zetah,\chih,\omegah,\xih)$,
with
$$
\begin{gathered}
\uuh \in H^1(0,\widehat{T};\bsV),\ \ \chih\in
L^\infty(0,\widehat{T};W)\cap
H^1(0,\widehat{T};V)\cap W^{1,\infty}(0,\widehat{T};H),
\\
  \zetah\in L^2(0,\widehat{T}; \bsY')\,,
  \ \
 \omegah\in L^\infty(0,\widehat{T}; H)\,,
\ \
\xih\in L^\infty (0,\widehat{T}; H) \,,
\end{gathered}
$$
such that the following convergences hold as $k\to\infty$
\begin{subequations}
\label{convs-eps}
\begin{align}
& \uu_{\eps_k}  \debole \uuh  &&  \text{ in }  H^1(0,\widehat{T}; \bsV), \label{conv1}\\
& \uu_{\eps_k} \to \uuh && \text{ in }  \mathrm{C}^0([0,\widehat{T}];
H^{1-s}(\Omega;\R^3)) \quad \text{for all } s\in (0,1),
\label{conv1bis}\\
 &\uu_{\eps_k}  \to \uuh  &&  \text{ in }  \mathrm{C}^0([0,\widehat{T}]; L^p(\GC;\R^3)) \quad \text{for all $1\leq p<4$},
\label{convf1}\\
&\chi_{\eps_k}  \debolestar \chih  &&  \text{ in }
L^{\infty}(0,\widehat{T};W) \cap
H^1(0,\widehat{T};V)  \cap W^{1,\infty}(0,\widehat{T};H),  \label{conv3}\\
&\chi_{\eps_k}  \to \chih && \text{ in }
\mathrm{C}^0([0,\widehat{T}];H^{2-s}(\GC))  \quad \text{for all } s\in (0,2),  \label{convf2}\\
&\zzeta_{\eps_k}\debole \ \zetah &&
\text{ in } L^2(0,\widehat{T};\bsY'), \label{conv2}\\
&\beta_{\eps_k}(\chi_{\eps_k})  \debolestar \xih  &&  \text{ in }
L^{\infty}(0,\widehat{T};H), \label{conv4}\\
&\rho_{\eps_k}(\dt\chi_{{\eps_k}})  \debolestar \omegah && \text{ in } \
L^{\infty}(0,\widehat{T};H)\,. \label{conv5}
\end{align}
\end{subequations}
Besides, $(\uuh,\zetah,\chih,\omegah,\xih)$ is a solution of
Problem \ref{prob:PDE},  fulfilling the energy-dissipation inequality \eqref{enid-balance}
(holding as a \emph{balance} if $\widehat{\rho}$ is $1$-positively homogeneous),  on $(0,\widehat{T})$.
\end{proposition}
% In the following calculations, we let $\epsi$ vary, say, in $(0,1)$ and, to simplify notation,
%we will directly work on the interval $(0,T)$ instead of $(0,\widehat{T})$.
 \begin{proof}
\textbf{Step $1$: a priori estimates.} We derive
  some {\it a priori}
estimates for (suitable norms of)  the family $(\uu_{\epsi},\zzeta_{\epsi}, \chi_{\epsi})_\eps$
on the interval $(0,\widehat{T})$. First of all, recall that the solution components $\chi_\eps$ are  found as fixed points of the Schauder operator $\mathcal{T}$ from \eqref{sol-op-comp} in the ball $\mathcal{S}_M(\widehat{T})$. Therefore,
in view of
 estimate \eqref{est-u-fix-pt}   there exists a
positive constant~$C$, independent of $\epsi$,  such that
\begin{equation}
\|\uu_\epsi\|_{H^1(0,\widehat{T};\bsV)}\leq C\,.
\label{st1}
\end{equation}
Then, \eqref{est-chi-fix-pt} (cf.\ also \eqref{est-Tchi}) yields
\begin{equation}
 \|\chi_\epsi\|_{
L^2 (0,\widehat{T}; W) \cap L^{\infty}(0,\widehat{T};V) \cap H^1(0,\widehat{T};H)}\leq C\,.  \label{st2}
\end{equation}
\par
Secondly, we perform a  comparison argument in  \eqref{eqU-approx}.
Observe that, by the second of \eqref{stima-puntuale} and \eqref{st2}, we have
\begin{equation}\label{bound1}
\| \nlocss{\chi_\eps}\|_{L^\infty ((0,\widehat{T})\times \GC)}
\leq  C \|\chi_\eps\|_{L^\infty(0, \widehat{T};H)}\leq C'\,,
\end{equation}
Therefore,
 on account of \eqref{st1}--\eqref{bound1} % and \eqref{bound2},
 we obtain
\begin{equation}\label{st4}
\|\zzeta_\eps\|_{L^2(0,\widehat{T};\bsY')}\leq C\,.
\end{equation}
\par
Next, we establish a further {\it regularity} estimate on $\chi_\eps$.
 We multiply (\ref{eqCHI-eta}) by $\dt(-\Delta \chi_\eps+\beta_\eps(\chi_\eps))$
  and we integrate over
$\GC\times (0,t)$.
Observe that all the calculations below can be made rigorous, for the very same system
\eqref{PDE-approx}, by arguing with  difference quotients. 
We have
\begin{align}
&
\frac{1}{2}\|-\Delta \chi_\eps(t)+\beta_\eps(\chi_\eps(t))\|^2_{H}+ \|\nabla\dt\chi_\epsi\|^2_{L^2(0,t;H)}+
\int_0^t\!\!\int_{\GC}
\beta'_\eps(\chi_{\eps }) (\dt\chi_{\eps})^2 \dd x \dd r
\nonumber\\
&+ \int_0^t\!\!\int_{\GC} \rho'_\epsi(\partial_t \chi_{\epsi})
|\nabla(\partial_t \chi_{\epsi})|^2 \dd x \dd r
+ \int_0^t\!\!\int_{\GC}
\beta'_\eps(\chi_{\eps }) \rho_\eps(\dt\chi_{\eps})
\dt\chi_{\eps} \dd x \dd r\nonumber\\
 &
\leq
\frac12\|-\Delta \chi_0+\beta_\epsi(\chi_0)\|^2_{H}+ I_1 +I_2+I_3
\,,\label{stimareg}
\end{align}
where $I_i, i=1,2,3$,  will be defined and estimated in what follows.
We note that the
integral terms
  on the left-hand side of
(\ref{stimareg}) are non-negative, thanks to the monotonicity of
$\rho_\epsi$ and of $\beta_\epsi$  and to the fact that $\rho_\eps(0)=0$. Moreover,
integrating by parts and taking into account \eqref{hyp-gamma}, \eqref{ini-chi}, \eqref{st1}, and \eqref{st2} we have
\begin{align}
\label{int1}
I_{1} &=-\int_0^t\int_{\GC}  \gamma'(\chi_\eps)\, \dt(-\Delta \chi_\eps+\beta_\eps(\chi_\eps))\dd x \dd r
\nonumber\\&
= \int_0^t\int_{\GC}  \gamma{''}(\chi_\eps)\dt\chi_\eps\, (-\Delta \chi_\eps+
  \beta_\eps(\chi_\eps)) \dd x \dd r
   - \int_{\GC} \gamma'(\chi_\eps(t)) (-\Delta \chi_\eps (t)+\beta_\eps(\chi_\eps(t)) \dd x   \nonumber\\&
  +  \int_{\GC} \gamma'(\chi_0) (-\Delta \chi_0+\beta_\eps(\chi_0))\dd x
\leq c \int_0^t \| \dt \chi_\eps \|_{H}   \|-\Delta \chi_\eps+
\beta_\eps(\chi_\eps)\|_{H} \dd r + c \|\chi_\eps(t)\|^2_{H}
\nonumber\\&+
\frac18 \| -\Delta \chi_\eps(t)+
\beta_\eps (\chi_\eps(t))\|^2_{H} +
c \|\chi_0 \|_{H} \, \|-\Delta \chi_0+\beta_\eps(\chi_0)\|_{H} +C
\nonumber\\
&  \leq c \int_0^t \| \dt \chi_\eps \|_{H} \| -\Delta\chi_\eps + \beta_\eps(\chi_\eps)\|_{H} \dd r +
\frac18 \| -\Delta \chi_\eps(t)+\beta_\eps (\chi_\eps(t))\|^2_{H} + C\,.
\end{align}
Arguing in a similar way and using %\eqref{hyp-k},
 \eqref{ini-u}, \eqref{ini-chi},
\eqref{stima-puntuale},
 \eqref{st1},  and \eqref{bound1},  we infer
\begin{equation}
\label{int2}
\begin{aligned}
I_{2} &=-\frac12\int_0^t\int_{\GC}  |\uu_\eps|^2 (1+\nlocss{\chi_\eps})\, \dt(-\Delta \chi_\eps+\beta_\eps(\chi_\eps))\dd x \dd r
\nonumber\\&
=\int_0^t\int_{\GC}  \uu_\eps\cdot\dt\uu_\eps (1+\nlocss{\chi_\eps})\,
(-\Delta \chi_\eps+\beta_\eps(\chi_\eps))\dd x \dd r
+\frac12\int_0^t\int_{\GC}  |\uu_\eps|^2 \nlocss{\dt\chi_\eps}\, (-\Delta \chi_\eps+\beta_\eps(\chi_\eps))\dd x \dd r
\nonumber\\&
\quad
-\frac12 \int_{\GC}  |\uu_\eps(t)|^2 (1+\nlocss{\chi_\eps}(t))\, (-\Delta \chi_\eps(t)+\beta_\eps(\chi_\eps(t)))\dd x
+\frac12 \int_{\GC}  |\uu_0|^2 (1+\nlocss{\chi_0})\, (-\Delta \chi_0+\beta_\eps(\chi_0))\dd x
\nonumber\\&
\leq c \int_0^t \|\uu_\eps \|_{L^4(\GC)} \, \| \dt\uu_\eps \|_{L^4(\GC)}
\|-\Delta\chi_\eps+\beta_\eps(\chi_\eps)\|_{H} \dd r
+c \int_0^t \|\dt\chi_\eps \|_{L^1(\GC)} \int_{\GC} |\uu_\eps|^2
\, |-\Delta \chi_\eps+\beta_\eps(\chi_\eps)| \dd x \dd r
\nonumber\\&
\quad +\frac18 \| -\Delta \chi_\eps(t)+\beta_\eps (\chi_\eps(t))\|^2_{H}+ c \| \uu_\eps(t)\|^4_{L^4(\GC)}+
c \|\uu_0 \|^2_{L^4(\GC)} \, \|-\Delta \chi_0+\beta_\eps(\chi_0)\|_{H}
\nonumber\\&
\leq c \int_0^t \|\uu_\eps \|_{L^4(\GC)} \, \| \dt\uu_\eps \|_{L^4(\GC)}
\|-\Delta \chi_\eps+\beta_\eps(\chi_\eps)\|_{H} \dd r
+c \int_0^t \|\dt\chi_\eps \|_{L^1(\GC)} \|\uu_\eps\|^2_{L^4(\GC)}
\,\| -\Delta \chi_\eps+\beta_\eps(\chi_\eps)\|_H\dd r
\nonumber\\&
\quad + \frac18 \|-\Delta \chi_\eps(t)+ \beta_\eps (\chi_\eps(t))\|^2_{H}+ c
\\&
\leq c \,\|\uu_\eps\|_{L^\infty(0,T;\bsW)} \int_0^t  \| \dt\uu_\eps \|_{L^4(\GC)}
\|-\Delta \chi_\eps+\beta_\eps(\chi_\eps)\|_{H} \dd r
\nonumber\\&
\quad +c\, \|\uu_\eps\|^2_{L^\infty(0,T;\bsW)} \int_0^t \|\dt\chi_\eps \|_{L^1(\GC)}
\,\| -\Delta \chi_\eps+\beta_\eps(\chi_\eps)\|_H\dd r
+ \frac18 \|-\Delta \chi_\eps(t)+ \beta_\eps (\chi_\eps(t))\|^2_{H}+ c\,.
\end{aligned}
\end{equation}
Prior to estimating $I_3$, we observe that, again by the second of \eqref{stima-puntuale}
combined with \eqref{st1}--\eqref{st2}, we have
\begin{align}\label{bound2}
\| \nlocss{\chi_\eps |\uu_\eps|^2}\|_{L^\infty((0,\widehat{T})\times \GC)}
\leq
 c \|\chi_\eps\|_{L^\infty(0,T;H)}\,\|\uu_\eps\|^2_{L^\infty(0,T;\bsW)} \leq C\,.
\end{align}
Then, exploiting %\eqref{hyp-k},
\eqref{ini-u}, \eqref{ini-chi},
\eqref{stima-puntuale},
\eqref{st1}--\eqref{st2}, and \eqref{bound2}, we get
\begin{align}
\label{int3}
I_{3} &=-\frac12\int_0^t\int_{\GC} \nlocss{\chi_\eps \,|\uu_\eps|^2}\,
\dt(-\Delta \chi_\eps+\beta_\eps(\chi_\eps))\dd x \dd r
\nonumber\\&
=\int_0^t\int_{\GC}   \nlocss{\chi_\eps \,\uu_\eps\cdot\dt\uu_\eps})\, (-\Delta \chi_\eps+
\beta_\eps(\chi_\eps))\dd x \dd r
+\frac12\int_0^t\int_{\GC} \nlocss{\dt\chi_\eps \, |\uu_\eps|^2 }\, (-\Delta \chi_\eps+ \beta_\eps(\chi_\eps))\dd x \dd r
\nonumber\\&
-\frac12 \int_{\GC}  \nlocss{\chi_\eps\, |\uu_\eps|^2 }(t)\, (-\Delta \chi_\eps(t)+\beta_\eps(\chi_\eps(t)))\dd x
+\frac12 \int_{\GC}   \nlocss{\chi_0\, |\uu_0|^2}\, (-\Delta \chi_0+\beta_\eps(\chi_0))\dd x
\nonumber\\&
\leq c \int_0^t \|\chi_\eps\, \|_{H}\|\uu_\eps \|_{L^4(\GC)} \, \| \dt\uu_\eps \|_{L^4(\GC)} \int_{\GC}  |-\Delta \chi_\eps+\beta_\eps(\chi_\eps)| \dd x \dd r
\nonumber\\&
+c \int_0^t \|\dt\chi_\eps \|_{H} \|\uu_\eps \|^2_{L^4(\GC)}
\int_{\GC}
|-\Delta \chi_\eps+\beta_\eps(\chi_\eps)|\dd x \dd r
\nonumber\\&
+\frac18 \|-\Delta \chi_\eps+ \beta_\eps (\chi_\eps(t))\|^2_{H}
%+
 %c\|\chi_\eps\|^2_{L^\infty(0,T;H)}\, \|\uu_\eps\|^4_{L^\infty(0,T;\bsW)}
 + \|\chi_0\|_H\, \|\uu_0\|^2_{\bsW} \,\|-\Delta \chi_0+\beta_\eps(\chi_0)\|_{L^1(\GC)}  +C 
\nonumber\\&
\leq c \|\chi_\eps\|_{L^\infty(0,T;H)}\, \|\uu_\eps\|_{L^\infty(0,T;\bsW)}
\int_0^t
 \| \dt\uu_\eps \|_{L^4(\GC)} \| -\Delta \chi_\eps+\beta_\eps (\chi_\eps)\|_{H} \dd r
\nonumber\\&
+c  \|\uu_\eps\|^2_{L^\infty(0,T;\bsW)}
\int_0^t
 \| \dt\chi_\eps \|_{H} \|-\Delta \chi_\eps+ \beta_\eps (\chi_\eps)\|_{H} \dd r
+\frac18 \| -\Delta \chi_\eps(t)+\beta_\eps (\chi_\eps(t))\|^2_{H}+C\,.
\end{align}
We plug the above estimates for $I_{i}, i=1,2,3$  into
\eqref{stimareg}, take into account  the previously obtained \eqref{st1}, \eqref{st2}, and apply
the Gronwall Lemma. In this way, we conclude that
\begin{equation}
\label{st3}\|-\Delta\chi_\epsi+\beta_\epsi(\chi_\epsi)\|_{L^\infty(0,\widehat{T};H)} \leq C\,.
\end{equation}
On  the other hand, by the monotonicity of $\beta_\epsi$, we have
for a.e. $t \in (0,\widehat{T})$
$$
\| (-\Delta\chi_\epsi+\beta_\epsi(\chi_\epsi))(t)
\|_{H}^2 \geq \|-\Delta\chi_\epsi(t)\|_{H}^2
+ \|\beta_\epsi(\chi_\epsi)(t)\|_{H}^2.
$$
Hence, by \eqref{st2} and
 the aforementioned
elliptic regularity results valid on $\GC$, 
%relying on the assumption that $\GC$ has $\mathrm{C}^{1,1}$-boundary
%\footnote{\RRR basta questa, giusto? Non serve il bordo $\mathrm{C}^2$, giusto? Ricordiamoci di ricontrollare..}
 we deduce that
\begin{equation}
\|\chi_\epsi\|_{H^1(0,\widehat{T};V)\cap
L^{\infty}(0,\widehat{T};W)}\leq C
\label{st4-new}
\end{equation}
and, besides,
\begin{equation}
\|\beta_\epsi(\chi_\epsi)\|_{L^{\infty}(0,\widehat{T};H)}\leq
C\,,\label{st5}
\end{equation}
for some positive constant $c$ independent of $\epsi$. Finally, a
comparison in (\ref{eqCHI-eta}), on account of (\ref{st1}),
(\ref{st4-new}), and  (\ref{st5}),  entails
\medskip\noindent
\begin{equation}
\|\rho_\epsi(\dt \chi_\epsi)\|_{L^{\infty}(0,\widehat{T};H)} \leq C\,.\label{st6}
\end{equation}
\paragraph{\bf Step $2$: compactness argument.}
Thanks to the above estimates, by well-known weak and weak*
compactness results, we deduce that, for every  vanishing  sequence $(\eps_k)_k$  there exists a (not relabeled) subsequence of
$\{(\uu_{\eps_k}, \zzeta_{\eps_k}, \chi_{\eps_k})\}_k$ such that  convergences
\eqref{conv1}, \eqref{conv3}, and \eqref{conv2}--\eqref{conv5} hold
as $\epsi_k\downarrow 0$. Moreover, using
 Aubin-Lions
compactness arguments  
and the generalized Ascoli theorem
(cf., e.g.,
 \cite{Simon87}), we also obtain \eqref{conv1bis},
 \eqref{convf1} and  \eqref{convf2},
entailing that the pair $(\uuh, \chih)$ (see Prop. \ref{prop:loc-exis})  fulfills the initial
conditions $\uu(0)=\uu_0$ and $\chi(0)=\chi_0$.
\paragraph{\bf Step $3$: limit passage in the approximate system \eqref{PDE-approx}.}  We exploit   convergences
\eqref{convs-eps}. 
%results  \eqref{conv1}, \eqref{conv3}, and \eqref{conv2}-\eqref{conv5} combined with \eqref{conv1bis},
% \eqref{convf1}, and  \eqref{convf2}.
To simplify notation we omit the subindex $k$ and denote the sequences of approximate solutions directly by the index $\eps$. We first discuss the convergence of the non-local terms $\nlocss{\chi_\eps}$ and $\nlocss{\chi_\eps|\uu_\eps|^2}$. We mainly exploit the continuity properties of the operator $\mathcal{K}$ stated by Lemma \ref{lemmaK}. We first observe that \eqref{convf2} implies that $\chi_\eps\rightarrow\chih$ in $\mathrm{C}^0([0, \widehat{T}];L^1(\GC))$, so that we can deduce \begin{equation}\label{convnl1}
 \nlocss{\chi_\eps}\rightarrow \nlocss{\chih}\quad\hbox{in }\mathrm{C}^0([0,\widehat{T}];L^\infty(\GC)).
 \end{equation}
 We analogously proceed to deal with the term $\nlocss{\chi_\eps|\uu_\eps|^2}$. Indeed,  \eqref{convf1} and \eqref{convf2}  lead to the strong convergence $\chi_\eps|\uu_\eps|^2\rightarrow\chih |\uuh|^2$ in $\mathrm{C}^0([0, \widehat{T}];L^1(\GC))$, % Note that
 since by Sobolev embeddings $\chi_\eps$ strongly converges in $\mathrm{C}^0([0, \widehat{T}];L^p(\GC))$ for any $1\leq p<\infty$, while $|\uu_\eps|^2\rightarrow|\uuh|^2$ in $\mathrm{C}^0([0,\widehat{T}];L^q(\GC))$ for all  $1\leq q<2$. Thus, we can deduce
   \begin{equation}\label{convnl2}
 \nlocss{\chi_\eps |\uu_\eps|^2}\rightarrow\nlocss{\chih|\uuh|^2}\quad\hbox{in }\mathrm{C}^0([0,\widehat{T}];L^\infty(\GC)).
 \end{equation}
Finally, \eqref{convnl1} and \eqref{convf1} with \eqref{bound1}, yield
\begin{equation}\label{convnl3}
|\uu_\eps|^2\nlocss{\chi_\eps}\debole |\uuh|^2\nlocss{\chih} \hbox{ in } L^2(0,\widehat{T} ;H),\quad|\uu_\eps|^2\nlocss{\chi_\eps}\rightarrow|\uuh|^2\nlocss{\chih} \hbox{ in } L^2(0, \widehat{T};L^p(\GC))  \text{ for all } 1\leq p<2.
\end{equation}
Analogously, we have that   %\eqref{st1}, \eqref{bound1}, \eqref{st4-new} (and trace theorems) combined with
 \eqref{convnl1} and \eqref{convf1}, \eqref{convf2}  imply (at least) that
\begin{equation}\label{convnl4}
\chi_\eps\uu_\eps\nlocss{\chi_\eps}\rightarrow \chih\uuh\nlocss{\chih}\quad\hbox{in } \mathrm{C}^0([0,\widehat{T} ];H).
\end{equation}
Now, it is a standard matter to pass to the limit (weakly) in the momentum balance
\eqref{eqU-approx} and in the flow rule \eqref{eqCHI-eta}
 and conclude that $(\uuh,\chih,\zetah,\xih,\widehat w)$ fulfill \eqref{eqU} and   \eqref{eqCHI}.  Actually, to complete our proof it remains to  identify $(\zetah,\xih,\widehat w)$. We point out that
 \eqref{convf2} and
 \eqref{conv4} %by semicontinuity
 lead to
 \[
 \limsup_{\eps \searrow0} \int_0^t \int_{\GC} \beta_\eps(\chi_\eps) \chi_\eps \dd x \dd r \leq \int_0^t \int_{\GC} \xih  \chih \dd x \dd r,
 \]
 for all $t\in (0,\widehat{T}]$,
 whence
 the identification $\xih\in\beta(\chih)$ a.e.\ in $\GC \times (0,\widehat{T})$. 
  Then, we aim to prove that
  \begin{align}
  &
  \label{ident-1}
  \zetah\in\balpha(\uuh)  && \text{in } \bsY' \  \aein \  (0,\widehat{T}),
  \\
  &
  \label{ident-2}
  \widehat w\in\rho(\partial_t\chih)   && \aein \  \GC \times  (0,\widehat{T}).
  \end{align}
   We first prove that
\begin{equation}
\limsup_{\varepsilon\searrow0}\int_0^t\langle\zzeta_\eps,\uu_\eps  \rangle_{\bsY} \dd r \leq \int_0^t\langle\zetah,\uuh  \rangle_{\bsY} \dd r,
\end{equation}
 whence \eqref{ident-1}. 
Let us  test \eqref{eqU-approx}
 by $\uu_\eps$ and integrate over $(0,t)$. This gives
 \begin{align}
 &\int_0^t\langle\zzeta_\eps,\uu_\eps  \rangle_{\bsY} \dd r =-\frac 1 2 b(\uu_\eps(t),\uu_\eps(t))+\frac 1 2 b(\uu_0,\uu_0)\\\no
 &-\int_0^t a(\uu_\eps,\uu_\eps) \dd r -\int_0^t\int_{\GC}\chi_\eps|\uu_\eps|^2 \dd x \dd r -\int_0^t\int_{\GC}\chi_\eps\uu_\eps^2\nlocss{\chi_\eps} \dd x \dd r +\int_0^t\langle F,\uu_\eps\rangle \dd r.
 \end{align}
 Taking the $\limsup$ as $\eps\searrow0$, exploiting the lower semicontinuity of the bilinear forms $a$ and $b$, as well as  \eqref{conv1}--\eqref{convf2}, \eqref{convnl2}, and \eqref{nl-symm}, we can infer that
 \begin{align}
 &\limsup_{\eps\searrow0}\int_0^t\langle\zzeta_\eps,\uu_\eps  \rangle_{\bsY} \dd r \leq -\frac 1 2 b(\uuh(t),\uuh(t))+\frac 1 2 b(\uu_0,\uu_0)\\\no
 &-\int_0^t a(\uuh,\uuh)  \dd r -\int_0^t\int_{\GC}\chih|\uuh|^2 \dd x \dd r -\int_0^t\int_{\GC}\chih\uuh^2\nlocss{\chih} \dd x \dd r +\int_0^t\langle F,\uuh\rangle \dd r = \int_0^t\langle\zetah,\uuh\rangle_{\bsY} \dd r ,
 \end{align}
 as desired. For the above inequality, we have  in particular used 
$$
-\int_0^t\int_{\GC}\chi_\eps\uu_\eps^2\nlocss{\chi_\eps} \dd x \dd r \rightarrow-\int_0^t\int_{\GC}\chih\uuh^2\nlocss{\chih} \dd x \dd r
$$
due to  \eqref{nl-symm}, \eqref{convf1}, \eqref{convf2}, and \eqref{convnl2}.
Analogously we proceed by testing \eqref{eqCHI-eta} by $\partial_t\chi_\eps$ and integrate over $(0,t)$ to prove that
\begin{equation}\label{lim-sup2}
\limsup_{\eps\searrow0}\int_0^t\int_{\GC}\rho_\eps(\partial_t\chi_\eps)\partial_t\chi_\eps \dd x \dd r \leq\int_0^t\int_{\GC}\widehat w\partial_t\chih \dd x \dd r
\end{equation}
  whence \eqref{ident-2}.
We have 
\begin{equation}
\label{needed-later}
 \begin{aligned}
\int_0^t\int_{\GC} \rho_\eps(\partial_t\chi_{\eps})\partial_t\chi_{\eps} \dd x \dd r  & =-\|\partial_t\chi_{\eps}\|^2_{L^2(0,t;H)}-
\frac{1}{2}\|\nabla\chi_\eps(t)\|^2_{H}+
\frac{1}{2}\|\nabla\chi_0\|^2_{H}
\\
 &
-\int_{\GC} \widehat\beta_\eps(\chi_{\eps}(t)) \dd x +\int_{\GC}
\widehat\beta_\eps(\chi_0) \dd  x -\int_{\GC} \gamma(\chi_{\eps}(t)) \dd x +\int_{\GC}
\gamma(\chi_0) \dd x
\\
&+\int_0^t\int_{\GC}
\left(-\frac{1}{2}|\uu_\eps|^2-\frac 12|\uu_\eps|^2\nlocss{\chi_\eps}-\frac 1 2 \nlocss{\chi_\eps|\uu_\eps|^2}\right)\partial_t\chi_{\eps} \dd x \dd r\,.
\end{aligned}
\end{equation}
First, we recall that by Mosco convergence of $\widehat\beta_\eps$  to $\widehat\beta$,  we have 
\begin{equation}
\liminf_{\eps\searrow 0} \int_{\GC}
\widehat\beta_\eps(\chi_{\eps}(t)) \dd x \geq\int_{\GC} \widehat\beta(\chih(t)) \dd x.
\end{equation}
Thus, taking the limsup as $\eps\searrow 0$  of both sides of \eqref{needed-later}   and   exploiting   \eqref{hyp-gamma}, \eqref{conv1}--\eqref{convf2}, \eqref{convnl2}, \eqref{convnl3} and the lower semicontinuity properties of
the   Lebesgue  norms, we get
 \begin{align}
& \limsup_{\eps \searrow 0}  \int_0^t\int_{\GC} \rho_\eps(\partial_t\chi_{\eps})\partial_t\chi_{\eps} \dd x \dd r \leq -\|\partial_t\chih\|^2_{L^2(0,t;H)}-
\frac{1}{2}\|\nabla\chih(t)\|^2_{H}+
\frac{1}{2}\|\nabla\chi_0\|^2_{H}
 \nonumber\\
 &
-\int_{\GC} \widehat\beta(\chih(t))\dd x +\int_{\GC}
\widehat\beta(\chi_0)\dd x -\int_{\GC} \gamma(\chih(t)) \dd x +\int_{\GC}
\gamma(\chi_0)\dd x \\\no
&+\int_0^t\int_{\GC}
\left(-\frac{1}{2}|\uuh|^2-\frac 12|\uuh|^2\nlocss{\chih}-\frac 1 2 \nlocss{\chih|\uuh|^2}\right)\partial_t\chih \dd x \dd r =\int_0^t\int_{\GC}\widehat w\partial_t\chih \dd x \dd r ,
\end{align}
 which gives  \eqref{lim-sup2}.
\par
The energy-dissipation inequality \eqref{enid-balance} follows from testing the momentum balance \eqref{eqU} by $\uu_t$, the flow rule \eqref{eqCHI} by $\chi_t$, integrating along the time interval $(0,\widehat{T})$, and  adding  the above relations. The key ingredients
to deduce \eqref{enid-balance}
are the chain rules
$\int_s^t a(\uu,\uu_t) \dd r = \tfrac12 a(\uu(t),\uu(t)) -  \tfrac12 a(\uu(s),\uu(s))$,
 \eqref{ch-rule-alpha}, and
\[
\begin{aligned}
&
\int_s^t \int_{\GC} \chi \uu \uu_t \dd x \dd r  +\frac12\int_s^t \int_{\GC}  |\uu|^2 \chi_t  \dd x \dd r
=\frac12 \int_{\GC}\chi(t)|\uu(t)|^2 \dd x  -\frac12 \int_{\GC}\chi(s)|\uu_(s)|^2 \dd x,
\\
&
\begin{aligned}
&
\int_s^t \int_{\GC} \chi \uu \nlocss{\chi} \uu_t \dd x \dd r  +\frac12\int_s^t \int_{\GC}  |\uu|^2  \nlocss{\chi} \chi_t  \dd x \dd r
+\frac12\int_s^t \int_{\GC}   \nlocss{\chi  |\uu|^2} \chi_t  \dd x \dd r
\\ &
\stackrel{(1)}{=}\frac12 \int_{\GC}\chi(t)|\uu(t)|^2 \nlocss{\chi(t)} \dd x  -\frac12 \int_{\GC}\chi(s)|\uu(s)|^2 \nlocss{\chi(s)} \dd x,
\end{aligned}
\\
&
\int_s^t \int_{\GC} \xi \chi_t \dd x \dd r =\int_{\GC}\widehat{\beta}(\chi(t)) \dd x - \int_{\GC}\widehat{\beta}(\chi(s)) \dd x,
\end{aligned}
\]
where for (1) we have also  used the symmetry property \eqref{nl-symm}.
In particular, observe
the first and second chain-rule identities
  allow us to combine the contributions from the momentum balance with  those from the adhesive flow rule. We also use that
\begin{equation}
\label{just-inequality}
\int_s^t\int_{\GC} \omega\chi_t \dd x \dd r  \geq \int_s^t \int_{\GC} \widehat{\rho}(\chi_t) \dd x \dd r
\end{equation}
since $\widehat{\rho}(0)=0$.
Then, \eqref{enid-balance}  follows with straightforward calculations.
\par
Observe that, if $\widehat{\rho}$ is $1$-positively homogeneous, \eqref{just-inequality}
 in fact holds with an \emph{equality} sign. Therefore, \eqref{enid-balance}  is valid as a \emph{balance}.
\end{proof}

%%%%%%%%%%%%%%%%%%%%%%%%%%%%%%%%%%%%%%%%%%%%%%%%%%%%%%%%%%%%%%%%%%%%%%%%%%%%%%%%%%%%%%%%%%%%%%%%%%%%%%%%%%%%%%%%%%%%%%%%%%%%%%%%%

\section{From a local to a global solution: conclusion of the proof of Theorem \ref{th:main}}
\label{s:5}
The standard procedure
for extending  the  local solution of Problem \ref{prob:PDE} found with Prop.\ \ref{prop:loc-exis} to a global one would involve
deriving  suitable a priori estimates on (local) solutions  of Prob.\ \ref{prob:PDE} on a generic interval $(0,\sfT)$, $\sfT \in (0,T]$, and showing that such estimates are in fact \emph{independent} of the final time $\sfT$. This would allow us to consider some maximal extension of the local solution from Prop.\ \ref{prop:loc-exis} and conclude, by a classical contradiction argument, that it  must be defined on the whole interval $(0,T)$.
\par
In the present case, in order to carry out the maximal extension argument, it would be necessary to rely on  a global-in-time  estimate
(1) in $\mathrm{C}^0 ([0,\sfT];\bsV)$ for $\uu$; (2) in $L^\infty (0,\sfT;W) \cap H^1(0,\sfT;V) \subset \mathrm{C}_{\mathrm{weak}}^0([0,\sfT];W)$ for $\chi$, in accordance with the requirements $\uu_0\in \bsV$ and $\chi_0\in W$ on the initial data. However,
while for $\uu$ a global estimate in $H^1(0,\sfT;\bsV)$ directly follows from the
 energy-dissipation
inequality \eqref{enid-balance} (cf.\ Lemma \ref{l:enid-cons} ahead), from  \eqref{enid-balance}
we can derive  \emph{global} estimates for $\chi$ only in  the spaces $L^\infty (0,\sfT;V) \cap H^1(0,\sfT;H)$
(and, as a by product,  cf.\ Lemma \ref{l:solvability-chi},
in $L^2(0,\sfT;W)$, as well).  In turn, the enhanced estimate  $L^\infty (0,\sfT;W) \cap H^1(0,\sfT;V) $ seems to be obtainable only upon performing the (formally written, here) test by $\partial_t(-\Delta\chi+\xi)$, with $\xi \in \beta(\chi)$. As shown by the proof of Prop.\ \ref{prop:loc-exis},
such a test can be made rigorous only in the frame of the approximate system \eqref{PDE-approx}. Unfortunately, for \eqref{PDE-approx}  the basic ``energy'' estimates for $\uu$ in  $H^1(0,\sfT;\bsV)$ and for $\chi $ in $L^\infty (0,\sfT;V) \cap H^1(0,\sfT;H)$
 do not possess a global-in-time character, cf.\  Remark \ref{rmk:features-reg-prob} and again  the proof of Prop.\ \ref{prop:loc-exis},  essentially because the lack of the positivity constraint on $\chi$ does not allow one to control from below the energy $\calE(\uu(t),\chi(t))$ on the left-hand side of \eqref{enid-balance}.
 \par
 These technical difficulties related to the extension procedure  were already manifest in the framework of the adhesive contact system first investigated in \cite{BBR1}, where nonlocal effects were disregarded.
 In Sec.\ 5 therein, to overcome them we have developed a careful prolongation argument,
 where the key idea is to extend the ($\uu$- and $\chi$-components of a) local solution $(\uuh,\zetah,\chih,\omegah,\xih)$, along with its `approximability properties'. The latter play a key role in recovering the enhanced regularity estimate for the $\chi$-component  in $L^\infty (0,\sfT;W) \cap H^1(0,\sfT;V) $.
 \par
 In what follows, we will adapt the argument from  \cite{BBR1} to our own situation. To avoid overburdening the paper, we will explain the steps of the procedure but  omit several details and  often refer to the calculations in  \cite[Sec.\ 5]{BBR1}, which could indeed be repeated with minor changes in the present framework.
\paragraph{\bf Scheme of the proof of the extension procedure.}
For the intents and purposes of our extension argument, we slightly modify the terminology introduced in Sec.\ \ref{s:2} and call \emph{solution of  Problem \ref{prob:PDE}}
on an interval $(0,\sfT)$
 \begin{equation}
 \label{sfT-later}
\begin{gathered}
\text{any pair $(\uu,\chi)$ with $\uu\in H^1(0,\sfT;\bsV)$ and $\chi \in L^\infty(0,\sfT;W) \cap H^1(0,\sfT;V) \cap W^{1,\infty}(0,\sfT;H)$ such that}
\\
\text{there exist $(\zzeta,\omega,\xi) \in L^2(0,\sfT;\bsY') \times L^\infty(0,\sfT;H) \times  L^\infty(0,\sfT;H) $}
\\
\text{such that $(\uu,\zzeta,\chi,\omega,\xi)$ solve system \eqref{PDE} on $(0,\sfT)$};
\end{gathered}
\end{equation}
 we  will adopt the very same convention for the solutions of the approximate Problem \ref{prob:approx}.
 \par
Recall that, in Prop.\ \ref{prop:loc-exis} we have shown that, along a sequence $\eps_k \down 0$, the (unique) solutions $(\uu_{\eps_k},\chi_{\eps_k})$ of Prob.\ \ref{prob:approx} (with regularization parameter $\eps_k$)
converge to a solution $(\uuh,\chih)$ of Prob.\ \ref{prob:PDE} on the time interval $(0,\widehat{T})$. In the following lines, for a fixed $\sfT\in (0,T]$ we will denote by
$(\uu_{\eps_k}^\sfT,\chi_{\eps_k}^\sfT)$  the unique solution of Prob.\  $\ref{prob:approx}_{\eps_k}$ on $(0,\sfT)$: clearly, for $\sfT \geq \widehat{T}$, $(\uu_{\eps_k}^\sfT,\chi_{\eps_k}^\sfT)$ is the (unique) extension of $(\uu_{\eps_k},\chi_{\eps_k})$.
\par
We are now in a position to define the solution concept for Problem  \ref{prob:PDE}
 that we are going to extend on the whole interval $[0,T]$.
\begin{definition}
\label{def:approx-sol}
Let $\sfT\in (0,T]$. We call a solution (in the sense of \eqref{sfT-later})
$
(\uu^{\sfT},\chi^{\sfT})
$
of Problem  \ref{prob:PDE} on $(0,\sfT)$
an \emph{approximable solution}   if there exists a (not-relabeled) subsequence  of $(\eps_k)_k$ such that
the related sequence $(\uu_{\eps_k}^\sfT,\chi_{\eps_k}^\sfT)_k$
 of solutions of  Problem  $\ref{prob:approx}_{\eps_k}$ on $(0,\sfT)$
 converge to $(\uu^\sfT,\chi^{\sfT})$ in the sense
 \begin{equation}
 \label{convs-sense}
 \| \uu_{\eps_k}^\sfT - \uu^\sfT\|_{\mathrm{C}^0([0,\sfT];\bsV)} +  \| \chi_{\eps_k}^\sfT - \chi^\sfT\|_{\mathrm{C}^0([0,\sfT];V)}  \to 0 \quad \text{as } k \to \infty\,.
 \end{equation}
\end{definition}
It is immediate to check that, for all $\sfT\in (\widehat{T},T]$ the functions $(\uu^\sfT,\chi^\sfT)$ are a proper extension of $(\uuh,\chih)$.
\par
Let us introduce the set
\[
\mathscr{T}: = \{ \sfT\in (0,T]\, : \ \text{there exists an approximable solution } (\uu^{\sfT},\chi^\sfT) \text{ of Prob.\ \ref{prob:PDE} on } (0,\sfT)\}.
\]
The calculations from the proof of Prop.\ \ref{prop:loc-exis} reveal  that $ (\uu^{\sfT},\chi^\sfT) $ fulfill the energy-dissipation inequality
\eqref{enid-balance}
(as a balance if $\widehat\rho$ is $1$-positively homogeneous) along any sub-interval $[s,t]\subset [0,\sfT]$.
Clearly, $\widehat{T} \in \mathscr{T}$, which is thus non-empty.
We aim to show that
\begin{equation}
\label{target-extension}
T^*: = \sup \mathscr{T} = T\,.
\end{equation}
In this way, we will conclude that $(\uuh,\chih)$ extends to a global solution on (0,T), fulfilling
\eqref{enid-balance}.
\par
The proof of \eqref{target-extension} is in turn split in the following steps:
\begin{description}
\item[\textbf{Step $1$}]  we show that  the basic \emph{energy estimates} for approximable solutions $(\uu^\sfT,\chi^\sfT)$ hold with constants independent of the final time $\sfT$;
\item[\textbf{Step $2$}]  we deduce that  the local solution $(\uuh,\chih)$ admits an extension $(\uu^*,\chi^*)$ on $(0,T^*)$, with $\uu^* \in H^1(0,T^*;\bsV)$ and $\chi \in L^2(0,T^*;W) \cap L^\infty (0,T^*;V) \cap H^1(0,T^*;H)$;
\item[\textbf{Step $3$}]  we show that $(\uu^*,\chi^*)$ is a solution  of Problem  \ref{prob:PDE}
(in the sense of \eqref{sfT-later}) on
$(0,T^*)$;
\item[\textbf{Step $4$}]
we conclude \eqref{target-extension}
via a contradiction argument.
\end{description}
We conclude this section by addressing the single steps with slightly more detail.
\paragraph{\bf Step $1$:}  We have the following result.
\begin{lemma}
\label{l:enid-cons}
Assume \eqref{hyp:GC}, \eqref{ass-K}, \eqref{hyp-alpha}--\eqref{hyp-gamma},
and \eqref{hyp:forces}.
Then, there exists a positive function $Q_4$  such that, for any $\sfT>0$ and any solution $(\uu,\chi)$ of Problem \ref{prob:PDE} on $(0,\sfT)$ there holds
\[
\| \uu\|_{H^1(0,\sfT;\bsV)} + \|\chi\|_{L^2(0,\sfT;W) \cap L^\infty(0,\sfT;V) \cap H^1(0,\sfT;H)} \leq
Q_4 (\|\uu_0\|_{\bsV},  \widehat\balpha(\uu_0),  \|\chi_0\|_{V}, \|\widehat{\beta}(\chi_0)\|_{L^1(\GC)},
\|\mathbf{F}\|_{L^2(0,\sfT;\bsV')})\,.
\]
\end{lemma}
\begin{proof}[Sketch of the proof]
We consider the energy-dissipation inequality
\eqref{enid-balance}
 on the interval $(0,\sfT)$ and observe that  the last integral term on the right-hand side can be absorbed into the dissipative term
$\int_0^\sfT b(\uu_t,\uu_t) \dd r $ on the left-hand side, cf.\ \eqref{added-4-later}.
Therefore, taking into account that
\[
|\calE (\uu_0,\chi_0) | \leq C (\|\uu_0\|_{\bsV} + \widehat\balpha(\uu_0) +  \|\chi_0\|_{V} +  \|\widehat{\beta}(\chi_0)\|_{L^1(\GC)}),
\]
 from \eqref{enid-balance}
we gather the bounds
\[
\sup_{t\in [0,\sfT]} |\calE(\uu(t),\chi(t))| + \int_0^{\sfT} 2 \calR(\uu_t,\chi_t) \dd r \leq C,
\]
whence the estimates for $\uu$ and $\chi$ in $H^1(0,\sfT;\bsV)$ and $L^\infty(0,\sfT;V) \cap H^1(0,\sfT;H)$. The estimate for $\chi$ in $L^2(0,\sfT;W)$ follows from a comparison for $-\Delta\chi$ in the flow rule for $\chi$, cf.\ the proof of Lemma \ref{l:solvability-chi}.
\end{proof}
\paragraph{\bf Step $2$:}
We consider a family $(\uu^\sfT,\chi^\sfT)_{\sfT\in \mathscr{T}}$ of approximable solutions, and define
\[
\begin{aligned}
&
\widetilde{\uu}_{\sfT}: [0,T^*] \to \bsV \quad \text{by } \ \widetilde{\uu}_{\sfT}(t): = \begin{cases}
\uu^\sfT(t) & \text{if } t\in [0,\sfT],
\\
\uu^\sfT(\sfT) & \text{if } t\in (\sfT, T^*],
\end{cases}
\\
&
\widetilde{\chi}_{\sfT}: [0,T^*] \to V \quad \text{by } \ \widetilde{\chi}_{\sfT}(t): = \begin{cases}
\chi^\sfT(t) & \text{if } t\in [0,\sfT],
\\
\chi^\sfT(\sfT) & \text{if } t\in (\sfT, T^*].
\end{cases}
\end{aligned}
\]
We consider a sequence $(\sfT_m)_m \subset \mathscr{T}$ with $\sfT_m\uparrow T^*$ and extract a (not) relabeled subsequence $(\widetilde{\uu}_{\sfT_m}, \widetilde{\chi}_{\sfT_m})_m$ suitably converging to a pair $(\uu^*,\chi^*)$ defined on $[0,T^*]$, which extends $(\uuh,\chih)$.
\paragraph{\bf Step $3$:}
We show that
$\chi^* \in L^\infty(0,T^*;W) \cap H^1(0,T^*;V) \cap W^{1,\infty}(0,T^*;H)  $ and that $(\uu^*,\chi^*)$ is a solution of Problem \ref{prob:PDE}, in the sense that
\[
\exists\, (\zzeta^*,\omega^*,\xi^*)\in L^2(0,T^*;\bsY') \times L^\infty(0,T^*;H)  \times L^\infty(0,T^*;H)  \text{ s.t. } (\uu^*,\zzeta^*,\chi^*,\omega^*,\xi^*) \text{ fulfill \eqref{PDE} on } (0,T^*).
\]
The proof of  enhanced regularity for $\chi^*$ and of  the  existence of the triple $(\zzeta^*,\omega^*,\xi^*)$ relies on the very notion of \emph{approximable} solution. For each element of the sequence $(\widetilde{\uu}_{\sfT_m}, \widetilde{\chi}_{\sfT_m})_m$ from Step 2 we pick a sequence of solutions  to the approximate Problem \ref{prob:approx} on $(0,\sfT_m)$, converging to $(\widetilde{\uu}_{\sfT_m}, \widetilde{\chi}_{\sfT_m})$ for fixed $m\in \N$. With a diagonalization procedure we extract a subsequence converging to $(\uu^*,\chi^*)$. We perform the  \emph{regularity} estimates leading to the enhanced properties of $\chi^*$, and to the existence of the triple $(\zzeta^*,\omega^*,\xi^*)$, on the level of the approximate system, by repeating the very same calculations in the proof of Proposition \ref{prop:loc-exis}. We refer the reader to \cite[Sec.\ 5]{BBR1} for all the details of the argument.
\paragraph{\bf Step $4$:}
Suppose that $T^*<T$.
To obtain a contradiction, it is sufficient to  extend the approximable solution $(\uu^*,\chi^*)$, defined on the interval $[0,T^*]$, to an approximable solution on the interval $[0,T^*+\eta]$ for some $\eta>0$. The argument
%Weextend The contradiction argument
%leading to
%consists in
for this
 is completely analogous to the one developed in  \cite[Sec.\ 5]{BBR1} for the adhesive contact system without nonlocal effects. We once again refer to  \cite{BBR1} for all  details.
\par
In this way, we deduce  \eqref{target-extension} and thus  conclude the proof of Theorem \ref{th:main}. \QED

%%%%%%%%%%%%%%

                                %
                                %


\begin{thebibliography}{99}




\bibitem{BBR1}
E.~Bonetti, G.~Bonfanti, and R.~Rossi. \newblock Global existence
for a contact problem with adhesion.
\newblock {\em Math. Meth. Appl. Sci.}, 31, 1029--1064, 2008.



\bibitem{BBR2}
E. Bonetti, G. Bonfanti, and R. Rossi. \newblock Well-posedness and
long-time behaviour for a  model of contact with adhesion.
\newblock {\em
	Indiana Univ. Math. J.},  56, 2787--2820, 2007.
	
\bibitem{BBR3}	
E. Bonetti, G. Bonfanti, and R. Rossi. \newblock Thermal effects in adhesive contact: modelling and analysis.
'newblock{\em Nonlinearity}, 22, 2967--2731, 2009.
	
\bibitem{BBRfric}	E. Bonetti, G. Bonfanti, and R. Rossi. \newblock Analysis of a unilateral contact problem taking into account adhesion and friction.
\newblock{\em J. Differential Equations}, 235, 438-462, 2012.

\bibitem{BBRfrictemp}	E. Bonetti, G. Bonfanti, and R. Rossi. \newblock Analysis of a temperature-dependent model for adhesive contact with friction.
\newblock{\em  Phys.\ D}, 285, 29-42, 2014.


\bibitem{BBRen}
E. Bonetti, G. Bonfanti, and  R. Rossi. \newblock  Modeling via internal energy balance and analysis of adhesive contact with friction in thermoviscoelasticity.
\newblock {\em  Nonlinear Anal. Real World Appl.}, 22, 473--507,	 2015.

\bibitem{bfl}
G.\ Bonfanti, M.\ Fr\'emond, F.\ Luterotti.  Global solution to a nonlinear system for irreversible phase changes. {\em Adv.\ Math.\ Sci.\ Appl.}, 10, 1--24, 2000.

\bibitem{brezis73}
H.~Br\'ezis.
\newblock {\em Op\'erateurs Maximaux Monotones et Semi-groupes de Contractions
  dans les Espaces de Hilbert}.
\newblock Number~5 in North Holland Math. Studies. North-Holland, Amsterdam,
  1973.



\bibitem{eck05}
C. Eck, J. Ji{\v{r}}{\'{\i}}, and  M. Krbec.
\newblock \emph{Unilateral contact problems}. \newblock
Pure and Applied Mathematics (Boca Raton) Vol. 270,
Chapman \& Hall/CRC, Boca Raton, Florida, 2005.

\bibitem{sofonea2}
J.\ R.\ Fern\'andez, M.\ Shillor, and  M.\ Sofonea. \newblock Analysis and numerical simulations of a dynamic contact problem
with adhesion. \newblock {\em Mathematical and Computer Modelling}, 37, 1317--1333, 2003.


\bibitem{freddi-fremond}
F.\ Freddi and  M.\ Fr\'emond.
\newblock {Damage in domains and interfaces: a coupled predictive theory}.
\newblock  {\em J.\ Mech.\ Mater.\ Struct.}, 1, 1205-1233, 2006.

	
\bibitem{fre}
M.~Fr\'emond.
\newblock {\em Non-smooth Thermomechanics}.
\newblock Springer-Verlag, Berlin, 2002.


\bibitem{MRK06}
M.\ Ko\v{c}vara, A.\ Mielke, and T.\ Roub\'\i \v{c}ek:
 A rate-independent approach to the delamination problem. {\em Math. Mech. Solids}, 11 (2006),  423--447.




\bibitem{MRS}
 A.\ Mielke, R.\ Rossi, and G.\ Savar\'e. Nonsmooth analysis of doubly nonlinear evolution equations. {\em Calc. Var. Partial Differential Equations}, 46 (2013), 253--310.


\bibitem{MRT12}
 A.\ Mielke, T.\ Roub\'\i \v{c}ek, and M.\ Thomas: From damage to delamination in nonlinearly elastic materials at small strains. {\em  J.\ Elasticity} 109 (2012), 235--273.



\bibitem{Mordu}
B.\ S. Mordukhovich.  \emph{Variational analysis and generalized differentiation. I. Basic theory.}  Springer-Verlag, Berlin, 2006.


\bibitem{Point}
N.~Point.
\newblock Unilateral contact with adherence.
\newblock  {\em Math. Methods Appl. Sci.}, 10, 367--381, 1988.



\bibitem{Raous}
M.~Raous, L.~Cang\'emi, and M.~Cocu.
\newblock A consistent model coupling adhesion, friction, and unilateral contact.
\newblock {\em Comput. Methods Appl. Mech. Eng.}, 177, 383--399, 1999.

%\bibitem{roub1}
%T.\ Roub\'\i\v cek. \newblock  Thermodynamics of rate independent processes in
%viscous solids at small strains. \newblock {\em SIAM J. Math. Anal.},
%40, 256--297, 2010.


%\bibitem{roub2}
%T.\  Roub\'\i\v cek, M.\ Thomas, and  C.\ Panagiotopoulos. \newblock Stress-driven local-solution approach to quasistatic brittle delamination. Submitted. WIAS Preprint 1889, 2013.




%\bibitem{Shi-Shillor}
%P. Shi, M. Shillor.
%\newblock Existence of a solution to the $N$-dimensional
%problem of thermoelastic contact.
%\newblock {\em  Comm. Partial Differential%
%	Equations}, 9-10, 1597--1618, 1992.

\bibitem{RouSur}
T.\ Roub\'\i \v{c}ek,  M.\ Kru\v{z}\'\i k, and J.\ Zeman:
 {\em Delamination and adhesive contact models and their mathematical analysis and numerical treatment}. Mathematical methods and models in composites, 349--400, Comput. Exp. Methods Struct., 5, Imp. Coll. Press, London, 2014.


\bibitem{Simon87}
J.~Simon.
\newblock Compact sets in the space {$L^p(0,T; B)$}.
\newblock {\em Ann. Mat. Pura Appl. (4)}, 146, 65--96, 1987.


\bibitem{sofonea-han-shillor}
M. Sofonea, W. Han, and M. Shillor. \newblock {\em Analysis and
approximation of contact problems with adhesion or damage}. \newblock
Pure and Applied Mathematics (Boca Raton), 276. Chapman \& Hall/CRC,
Boca Raton, FL, 2006.


\bibitem{mon3}
 M.\ Shillor, M.\  Sofonea, J.J.\ Telega.
  \newblock \emph{Models and Analysis of Quasistatic Contact}.
  \newblock   Lecture Notes in Physics, 655. Springer, Berlin, 2004.




\end{thebibliography}
\end{document}